\documentclass[a4paper,10pt]{article}
\author{Aijin Lin\\School of Mathematical Sciences, Peking University\\ Beijing 100871, China\\
Email: linaj327@gmail.com}
\title{Index computations in FJRW theory}
\date{}
\usepackage{amsfonts}
\usepackage{amsthm}
\begin{document}
\maketitle
\renewcommand{\abstractname}{Abstract}
\begin{abstract}
We compute the index of the real Cauchy-Riemann operator defined in FJRW theory in case of the smooth metric. For the cylindrical metric, we study the relation between the index of the linearized operator of Witten map and weights in weighted Sobolev space.
\end{abstract}
\vspace{12pt}
\let\thefootnote\relax \footnotetext {The author's name Aijin Lin was earlier known as Ling Lin. I am deeply sorry for any inconvenience that I have brought to you.}
\newtheorem{theorem}{Theorem}[section]
\newtheorem{lemma}[theorem]{Lemma}
\section{Introduction}
\hspace{14pt} In this paper we mainly investigate two index problems coming from $FJRW$ theory (see [10, 11, 12]) constructed
by Huijun Fan, Tyler J. Jarvis, Yongbin Ruan based on a proposal of Edward Witten.\par
One is a concrete index computation problem of the real Cauchy-Riemann operator under the smooth metric, the other is about relation between the index of the linearized operator of Witten map [12] and weights in weighted Sobolev space in case of the cylindrical metric. They are both interesting and also important for $FJRW$ theory.\par
First of all let us review the background of these problems. We adopt notations of [10] in the following discussion.\par
Let $W\in \mathbb{C}[x_{1},\cdots,x_{t}]$ be a quasi-homogeneous polynomial, i.e., there exists degrees $d,k_{1},\cdots,k_{t}\in \mathbb{Z}^{>0}$ such that for any $\lambda\in \mathbb{C}^{\ast}$\\
\begin{equation}
W(\lambda^{k_{1}}x_{1},\cdots,\lambda^{k_{t}}x_{t})=\lambda^{d}W(x_{1},\cdots,x_{t}).\\
\end{equation}
\\
$\textbf{Definition 1.1.}$  W is called $nondegenerate$ if\\
\\
(1) the fractional degrees $q_{i}=\frac{k_{i}}{d}$ are uniquely determined by $W$; and\\
(2) the hypersurface defined by $W$ in weighted projective space is non-singular, or
equivalently, the affine hypersurface defined by $W$ has an isolated singularity at
the origin.
\begin{lemma}([10])
If $W$ is nondegenerate, then the group
   \begin{equation}
   H:=\{(\alpha_{1},\cdots,\alpha_{t})\in (\mathbb{C}^{\ast})^{t}|W(\alpha_{1}x_{1},\cdots,\alpha_{t}x_{t})=W(x_{1},\cdots,x_{t})\}
   \end{equation}
   of diagonal symmetries of $W$ is finite. In particular, we have
 \begin{equation}     H\subseteq \mu_{d/k_{1}}\times\cdots\times \mu_{d/k_{t}}\cong k_{1}Z/d\times\cdots k_{t}Z/d,  \end{equation}\\
     where $\mu_{l}$ is the group of lth roots of unity.
\end{lemma}
Here finiteness of the group $H$ is necessary for later discussion.\\
\\
$\textbf{W-spin structures on smooth orbicurves}.$ Let $(\tilde{\Sigma},\textbf{z,m})$ be a smooth orbicurve (orbifold Riemann surface), i.e., $(\tilde{\Sigma},\textbf{z,m})$ is a Riemann surface $\Sigma$ with marked points $\textbf{z}=\{z_{i}\}$ having orbifold structure near each marked point $z_{i}$ given by a faithful action of $\mathbb{Z}/m_{i}$.\par In other words, a neighborhood of each marked point is uniformized by the branched covering map $z\rightarrow z^{m_{i}}$.\par
Let $\rho:\tilde{\Sigma}\rightarrow\Sigma$ be the natural projection to the coarse Riemann surface $\Sigma$.
A line bundle $L$ on $\Sigma$ can be uniquely lifted to an orbifold line bundle on $\tilde{\Sigma}$. We denote the lifted bundle by the same $L$.
\\
\\
$\textbf{Definition 1.2.}$  Let $K$ be the canonical bundle of $\Sigma$, and let \\
 \begin{equation}       K_{log}:=K\otimes \mathcal{O}(z_{1})\otimes\cdots \otimes \mathcal{O}(z_{k})  \end{equation}\\
be the $log$-$canonical$ $bundle$. The holomorphic sections are holomorphic 1-forms away from the special points $\{z_{i}\}$ and with at simple poles at the $z_{i}$. $K_{log}$ can be thought of as the canonical bundle of the punctured Riemann surface $\Sigma-\{z_{1},\cdots,z_{k}\}.$ Suppose that $L_{1},\cdots,L_{t}$ are orbifold line bundles on $\tilde{\Sigma}$ with isomorphisms
 $\varphi_{j}:W_{j}(L_{1},\cdots,L_{t})\widetilde{\rightarrow}K_{log}$,where by $W_{j}(L_{1},\cdots,L_{t})$ we mean the $j$th monomial of $W$ in $L_{i}$, \begin{equation}
       W_{j}(L_{1},\cdots,L_{t})=L^{\otimes b_{1j}}_{1}\otimes\cdots,\otimes L^{\otimes b_{tj}}_{t} . \end{equation}

Here $K_{log}$ is identified with its pull-back to $\tilde{\Sigma}$. \\
The tuple $(L_{1},\cdots,L_{t},\varphi_{1},\cdots,\varphi_{s})$ is called a $W$-$spin$ $structure.$\\
\\
$\textbf{Definition 1.3.}$ Suppose that the chart of $\tilde{\Sigma}$ at an orbifold point $z_{i}$ is $D/(\mathbb{Z}/m)$ with action $e^{\frac{2\pi i}{m}}(z)=e^{\frac{2\pi i}{m}}z$. Suppose that the local trivialization of an orbifold line bundle $L$ is $(D\times \mathbb{C})/(\mathbb{Z}/m)$ with the action  \begin{equation}
  e^{\frac{2\pi i}{m}}(z,w)=(e^{\frac{2\pi i}{m}}z,e^{\frac{2\pi i\nu}{m}}w). \end{equation}
When $\nu=0$, we say that $L$ is $Broad$ at $z_{i}$; when $\nu>0$, we say $L$ is $Narrow$ at $z_{i}$.\par
A $W$-spin structure $(L_{1},\cdots,L_{t},{\varphi}_{1},\cdots,{\varphi}_{s})$ is called $Broad$ at the point $z_{i}$ if the group element $h=(exp(2\pi i{\nu}_{1}/m),\cdots,exp(2\pi i{\nu}_{t}/m))$ defined by the orbifold action on the line bundles $L_{k}$ at $z_{i}$ acts trivially on all the line bundles occurring in the monomial $W_{j}$. In other words, the $W$-spin structure is Broad if there is a monomial $W_{j}=c_{j}\prod x^{b_{l,j}}_{l}$ in $W$ such that for every $l$ with $b_{i,j}>0$ the line bundle $L_{l}$ is Broad at $z_{i}$.\\
\\
$\textbf{Desingularization}$. If $L$ is an orbifold line bundle on a smooth orbifold Riemann surface $\tilde{\Sigma}$, then the sheaf of locally invariant holomorphic sections of $L$ is locally free of rank one, and hence dual to a unique orbifold line bundle $|L|$ on $\Sigma$. We also denote $|L|$ by $\rho_{\ast}L$, and it corresponds to the $desingularization$ [8] of $L$. It can be constructed as follows.\par
We keep the local trivialization at other places and change it at the orbifold point $z_{i}$ by a $\mathbb{Z}/m$-equivariant map $\Psi:(D-\{0\})\times \mathbb{C}\rightarrow(D-\{0\})\times \mathbb{C}$ by \\
     \begin{equation} (z,w)\rightarrow(z^{m},z^{-\nu}w), \end{equation}
where $\mathbb{Z}/m$ acts trivially on the second $(D-\{0\})\times \mathbb{C}$. Then, we extend $L_{((D-\{0\})\times \mathbb{C})}$ to a smooth holomorphic line bundle over $\Sigma$ by the second trivialization. Since $\mathbb{Z}/m$ acts trivially, this gives a line bundle over $\Sigma$, which is $|L|$. Note that if $L$ is Broad at $z_{i}$, then $|L|=L$ locally. When $L$ is Narrow at $z_{i}$, then $|L|$ differs from $L$.\\
\\
$\textbf{Smooth metric and cylindrical metric}$ \hspace{10pt} We fix a $W$-spin structure \begin{eqnarray*}(L_{1},\cdots,L_{t},\varphi_{1},\cdots,\varphi_{s})\end{eqnarray*}
For each monomial $W_{i}$, let \begin{eqnarray*}D=-\sum^{k}_{l=1}\sum^{t}_{j=1}b_{ij}(a_{j}(h_{l})-q_{j})z_{l} \end{eqnarray*} be a divisor, where $a_{j}(h_{l})$ is the orbifold action on the line bundle $L_{j}$ at the marked point $z_{l}$ in the expression (6), then there is a canonical meromorphic section $s_{0}$ with divisor $D$. This section provides the identification  \begin{equation}
    s^{-1}_{0}:K_{\Sigma}\otimes \mathcal{O}(D)\cong K_{\Sigma}(D), \end{equation}
where $K_{\Sigma}(D)$ is the sheaf of local, possibly meromorphic, sections of $K_{\Sigma}$ with zeros determined by $D$. When at least one of the line bundles occurring in the monomial $W_{i}$ is Narrow at $z_{l}$, then $D$ is not effective. So the local section of $K_{\Sigma}(D)$ has zeros, and hence is a natural sub-sheaf of $K_{\Sigma}$. In general, however, it is a sub-sheaf of $K_{log}$. For each marked point, there is a canonical local section $\frac{dz}{z}$ of $K_{log}$. Using the isomorphism $\varphi_{i}$, there is a local section $t_{i}$ of $L_{i}$ with the property $W_{i}(t_{1},\cdots,t_{k})=\frac{dz}{z}.$ The choice of $t_{i}$ is unique up to the action of the group $H$ defined in Lemma 1.1. \par
We choose a metric on $K_{log}$ with the property $|\frac{dz}{z}|=\frac{1}{|z|}$.
It induces a unique metric on $L_{i}$, with property $|t_{i}|=|z|^{-q_{i}}$. Using the correspondence between $L_{i}$ and $|L_{i}|$, it induces a metric on $|L_{i}|$ with the behavior $|e_{i}|=|z|^{a_{i}(h)-q_{i}}$ near a marked point, where $e_{i}$ is the corresponding local section of $|L_{i}|$. This metric on $|L_{i}|$ is called $smooth$ $metric$. In particular, it is a singular metric  where $L$ is Broad ($a_{i}(h)=0$) at some marked point. \par
If we choose a metric on $K_{log}$ with the property $|\frac{dz}{z}|=1$.
 Using the correspondence between $L_{i}$ and $|L_{i}|$, it induces a metric on $|L_{i}|$ with the behavior $|e_{i}|=|z|^{a_{i}(h)}$ near a marked point, where $e_{i}$ is the corresponding local section of $|L_{i}|$. This metric on $|L_{i}|$ is called $cylindrical$ $metric$.\par
Let $(\Omega,z_{1},\cdots,z_{k})$ be an obicurve with k marked points, and $B_{1}(z_{l})$ be the unit closed disc with the center $z_{l}$. Choose a compact subset $\Omega\subset\Sigma\backslash\cup_{l=1}^{k} B_{e^{-1}}(z_{l})$ such that $\{\Sigma,B_{1}(z_{1}),\cdots,B_{1}(z_{k})\}$ can cover $\Sigma$. Let $\varphi_{0},\cdots,\varphi_{k}$ be a set of partition functions subordinate to the cover. Let $e_{j}$ be basis of orbifold line bundle $L_{j}$ on $\Sigma$, the smooth metric is defined above: $|e_{j}|=|z|^{a_{j}(h_{l})-q_{j}}$. \par
Let the section of $L_{j}$ on $B_{1}(z_{1})$ be $u_{j}=\tilde{u}_{j}e_{j}$, we can define norms of $L^{p},L_{1}^{p}$ as follows: \begin{eqnarray*}
||u_{j}||_{p;B_{1}(z_{l})}&=&({\int_{B_{1}(z_{l})}|\tilde{u}_{j}|^{p}|{e_{j}}|^{p}|dzd\bar{z}|})^{1/p} \\ ||u_{j}||_{1,p;B_{1}(z_{l})}&=&({\int_{B_{1}(z_{l})}(|\tilde{u}_{j}|^{p}+|\partial \tilde{u}_{j}|^{p}+|\bar\partial \tilde{u}_{j}|^{p})|{e_{j}}|^{p}|dzd\bar{z}|})^{1/p}  \end{eqnarray*}
\par In the inner part of $\Omega$ which is away from the marked points, the norm is defined by the standard Sobolev norm $||u_{j}||_{W_{k}^{p}(\Omega)}$.\par
 The global $L^{p},L_{1}^{p}$ norms are defined as : \begin{eqnarray*}
||u_{j}||_{p}=||\varphi_{0}u_{j}||_{W_{0}^{p}(\Omega)}+\Sigma_{l=1}^{k}||\varphi_{l}u_{j}||_{p;B_{1}(z_{l})},\\
||u_{j}||_{1,p}=||\varphi_{0}u_{j}||_{W_{1}^{p}(\Omega)}+\Sigma_{l=1}^{k}||\varphi_{l}u_{j}||_{1,p;B_{1}(z_{l})}, \end{eqnarray*} \par
The weighted Sobolev space $L_{1}^{p}(\Sigma,|L_{j}|)$ is defined as the closure of $C_{0}^{\infty}(\Sigma\setminus(z_{1},\cdots,z_{k}),|L_{j}|)$ under the norm $||\cdot||_{1,p}$, and $L^{p}(\Sigma,|L_{j}|\otimes\wedge^{0,1})$ is defined as closure of $C_{0}^{\infty}(\Sigma\setminus(z_{1},\cdots,z_{k}),|L_{j}|\otimes\wedge^{0,1})$ under norm $||\cdot||_{p}$, where $\wedge^{0,1}$ is the (0,1)-form space of the Riemann surface $\Sigma$. We can define the Cauchy-Riemann operator $\bar \partial:L_{1}^{p}(\Sigma,|L_{j}|)\rightarrow L^{p}(\Sigma,|L_{j}|\otimes\wedge^{0,1})$,  \par
For convenience, we make a coordinate transformation $z=e^{-t-i\theta}$, hence the punctured neighbourhood  $B_{1}(z_{l})\setminus \{z_{l}\}$ is transformed to an end :$S^{1}\times [0,\infty)$.\\
Likewise, the operator $\bar \partial$ is changed as ${\bar \partial}^{t,\theta}$, and their relation is as follows: \begin{eqnarray*}
\bar \partial=-\frac{1}{2}e^{t-i\theta}(\partial_{t}+i\partial_{\theta})=-e^{t-i\theta}{\bar \partial}^{t,\theta}  \end{eqnarray*}
The original problem is transformed as follows: \begin{eqnarray*}
{\bar \partial}^{t,\theta}:\hat{W}_{1,1+k_{j,l}}^{p}\rightarrow W_{0,1+k_{j,l}}^{p}  \end{eqnarray*}
where $k_{j,l}=-a_{j}(h_{l})+q_{j}-2/p$, and the spaces $\hat{W}_{1,1+k_{j,l}}^{p},W_{0,1+k_{j,l}}^{p}$ are defined as the closure of smooth section space $\Gamma(|L_{j}|_{B_{1}(z_{l})})$ under the norms $||\cdot||_{\hat{W}_{1,1+k_{j,l}}^{p}}, ||\cdot||_{{W}_{0,1+k_{j,l}}^{p}}$ respectively:
 \begin{eqnarray}
||u_{j}||_{\hat{W}_{1,1+k_{j,l}}^{p}}&=&(\int_{S^{1}\times [0,\infty)}|\tilde{u}_{j}|^{p}e^{k_{j,l}pt}+(|\partial\tilde{u}_{j}|^{p}+|\bar \partial\tilde{u}_{j}|^{p})e^{(1+k_{j,l})pt})^{1/p}\\
||u_{j}||_{{W}_{0,1+k_{j,l}}^{p}}&=&(\int_{S^{1}\times [0,\infty)}|\tilde{u}_{j}|^{p}e^{(1+k_{j,l})pt})^{1/p}  \end{eqnarray} \par
[10] has studied the index theory of $\bar \partial$ in case of the smooth metric, especially they got an index theorem as follows:\\
\begin{theorem}([10]) In case of the smooth metric, if $1<p<\frac{2}{q_{j}}$ and $a_{j}(h_{l})-q_{j}+\frac{2}{p}\neq1,2$ for any $l (l=1,\cdots,k)$, then $\bar{\partial}:L_{1}^{p}(\Sigma, |L_{j}|)\rightarrow L^{p}(\Sigma, |L_{j}|\otimes \wedge^{0,1})$ is a Fredholm operator.\par
In particular, if $2<p<\frac{2}{1-\bar{\delta_{j}}}$ we have the index relation  \begin{eqnarray*}
&&\verb"ind"(\bar{\partial}: L_{1}^{p}(\Sigma, |L_{j}|)\rightarrow L^{p}(\Sigma, |L_{j}|\otimes \wedge^{0,1}))\\&=&\verb"ind"(\bar{\partial}^{t,\theta}: W_{1,1+\kappa}^{p}\rightarrow W_{0,1+\kappa}^{p})+\sharp\{z_{l}: c_{jl}<0\} \end{eqnarray*}
and the index is independent of $p\in (2,\frac{2}{1-\bar{\delta_{j}}}$). Where $\bar{\delta_{j}}=min_{l:c_{jl}>0}(c_{jl})$, $c_{jl}=a_{j}(h_{l})-q_{j}$. \end{theorem}
For convenience, we call this theorem as the index transformation theorem. Using this theorem, concrete index computation of $\bar \partial$ is transformed to compute $\verb"ind"(\bar{\partial}^{t,\theta}: W_{1,1+\kappa}^{p}\rightarrow W_{0,1+\kappa}^{p})$. \par
Now we can state our first main theorem as follows:
\begin{theorem} In case of the smooth metric, we can compute the index of $\bar \partial$ as follows:
\begin{eqnarray*}
&&\verb"ind"(\bar{\partial}: L_{1}^{p}(\Sigma, |L_{j}|)\rightarrow L^{p}(\Sigma, |L_{j}|\otimes \wedge^{0,1}))\\
&=&k+(1-2q_{j})(2-2g-k)-2\sum_{l=1}^{k} a_{j}(h_{l})+\sharp\{z_{l}:c_{jl}<0\}
\end{eqnarray*}
where $c_{jl}=a_{j}(h_{l})-q_{j}$, $q_{j}$ is the fractional degree of a nondegenerate quasi-homogeneous polynomial $W$ with respect to the jth variable, $a_{j}(h_{l})=v_{j,l}/m_{l}$  is the orbifold action on the line bundle $L_{j}$ at the marked point $z_{l}$. \par
And the right side is actually an integer, i.e.
\begin{eqnarray*}k+(1-2q_{j})(2-2g-k)-2\sum_{l=1}^{k} a_{j}(h_{l})
+\sharp\{z_{l}:c_{jl}<0\}\in \mathbb{Z} \end{eqnarray*}  \end{theorem}

As for the cylindrical metric, we can similarly define the operator
\begin{eqnarray*} \bar{\partial}:L_{1}^{p}(\Sigma, |L_{j}|)\rightarrow L^{p}(\Sigma, |L_{j}|\otimes \wedge^{0,1})\end{eqnarray*}
only by replacing the smooth metric $|e_{j}|=|z|^{a_{j}(h_{l})-q_{j}}$ by the cylindrical metric $|e_{j}|=|z|^{a_{j}(h_{l})}$.\par
In [12], the authors introduced the linearized operator $D$ of the Witten map, that is:\begin{eqnarray*}
D=D_{\wp,\mu}WI:{L_{1}}^{p}(\Sigma,L_{1}\times L_{2}\times \cdots \times L_{N})\rightarrow L^{p}(\Sigma,L_{1}\otimes\wedge^{0,1})\times\cdots L^{p}(\Sigma,L_{N}\otimes\wedge^{0,1}) \end{eqnarray*}
Then they proved this operator is a Fredholm operator under some mild conditions in case of the cylindrical metric and computed its index.
If we add weight $\delta$ to Sobolev spaces above, and consider the following operator:\begin{eqnarray*}
D^{\delta}=D_{\wp,\mu}WI:{L_{1}}^{p,\delta}(\Sigma,L_{1}\times L_{2}\times \cdots \times L_{N})\rightarrow L^{p,\delta}(\Sigma,L_{1}\otimes\wedge^{0,1})\times\cdots L^{p,\delta}(\Sigma,L_{N}\otimes\wedge^{0,1}) \end{eqnarray*}
we can ask what is the relation between index$(D^{\delta})$ and $\delta$ ? \par
Our second main theorem can totally solve this problem. Let's state it as follows:
\begin{theorem} Assuming that $D^{\delta}$, $D^{\delta'}$ stand for the linearized operator $D$ of the Witten map with weights $\delta, \delta'$ respectively. In case of the cylindric metric, we have the index jumping formula
\begin{eqnarray*}
\verb"ind"(D^{\delta})-\verb"ind"(D^{\delta'})&=&\Sigma^{N}_{j=1}\Sigma^{k}_{l=1}([\delta_{j,l}]-[\delta'_{j,l}]).
\end{eqnarray*}
where $\delta=(\delta_{jl})\in \mathbb{R}^{N\times k},\delta'=(\delta'_{jl})\in \mathbb{R}^{N\times k}$ are weight matrixes, $N$ stands for the number of variable in a nondegenerate quasi-homogeneous polynomial $W$, $k$ stands for the number of marked points.
\end{theorem}
\par
The paper is organized as follows. In section 2, we will first review Riemann-Roch theorem with boundary, Donaldson index theory and Lockhat-McOwen theory, then as an application of their work we prove Theorem 1.3. In section 3, we will prove Theorem 1.4 by generalizing Theorem 1.2.
\section{Index computation in case of the smooth metric}

\hspace{14pt} In this section, we will prove Theorem 1.3, which is equivalent to compute
\begin{equation}\verb"ind"(\bar{\partial}: L_{1}^{p}(\Sigma, |L_{j}|)\rightarrow L^{p}(\Sigma, |L_{j}|\otimes \wedge^{0,1}))\end{equation}
in case of the smooth metric. By Theorem 1.2, this index problem can be transformed to compute
\begin{equation}\verb"ind"(\bar{\partial}^{t,\theta}: W_{1,1+\kappa}^{p}\rightarrow W_{0,1+\kappa}^{p})\end{equation}. \par
 Let $e_{j}$ be a basis of orbifold line bundle $L_{j}$ on $\Sigma$, and recall that the smooth metric is defined: $|e_{j}|=|z|^{a_{j}(h_{l})-q_{j}}$, where $q_{j}$ is the fractional degree of a nondegenerate quasi-homogeneous polynomial $W$ with respect to the jth variable, $a_{j}(h_{l})$ is the orbifold action on the line bundle $L_{j}$ at the marked point $z_{l}$. \par
Note that $q_{j}$ is a nonnegative rational number, and $a_{j}(h_{l})=0$  when $L_{j}$ is Broad (see section 1) at $z_{l}$. So $a_{j}(h_{l})-q_{j}<0$ may happen. In that case the marked point $z_{l}$ is a singularity with respect to the smooth metric. Therefore, the smooth metric is a singular metric and we can not directly compute the index of the operator $\bar{\partial}: L_{1}^{p}(\Sigma, |L_{j}|)\rightarrow L^{p}(\Sigma, |L_{j}|\otimes \wedge^{0,1})$.

Let $z_{l}$ be a marked point, and consider the restriction of the bundle $|L_{j}|_{B_{1}(z_{l})}$ on the disc $B_{1}(z_{l})$. Assume that $B_{1}(z_{l})\times \mathbb{C}\rightarrow|L_{j}|_{B_{1}(z_{l})}:(z,w)\rightarrow\Psi_{l}(z)w$ is a trivialization such that $\Psi_{l}(e^{i\theta})\mathbb{R} $ forms a totally real bundle on $S^{1}_{l}=\partial B_{1}(z_{l})$. Define spaces \\
 \\
$\hat{W}^{p,B}_{1,1+\kappa_{j,l}}:=\{\tilde{u}_{j}\in\hat{W}_{1,1+k_{j,l}}^{p}|$ and $\tilde{u}_{j}(e^{i\theta})\in \Psi_{l}(e^{i\theta})\mathbb{R}\}$\\
$W^{p,B}_{1,1+\kappa_{j,l}}:=\{\tilde{u}_{j}\in W_{1,1+k_{j,l}}^{p}|$ and $\tilde{u}_{j}(e^{i\theta})\in \Psi_{l}(e^{i\theta})\mathbb{R}\}$ \\
\\
where $\hat{W}_{1,1+k_{j,l}}^{p}, W_{1,1+k_{j,l}}^{p}$ are defined as (9),(10), $k_{j,l}=-a_{j}(h_{l})+q_{j}-2/p$.\par
We can also define the space  \\
\\
$W^{p,B}_{1}(inn):=\{\tilde{u}\in W^{p}_{1}(\Sigma \setminus\cup_{l}B_{1}(z_{l}))|\tilde{u}(e^{i\theta})\in\Psi_{l}(e^{i\theta})\mathbb{R}$ for   $e^{i\theta}\in S^{1}_{l}\}$ \\
\\
where $S^{1}_{l}$ is the boundary of $B_{1}(z_{l})$. \par
Now consider local index problems
$\verb"ind"(\bar{\partial}^{t,\theta}:W^{p,B}_{1}(inn)\rightarrow W^{p}_{0}(inn))$, $\verb"ind"(\bar{\partial}^{t,\theta}:\hat{W}^{p,B}_{1,1+\kappa_{j,l}}\rightarrow W^{p}_{0,1+\kappa_{j,l}}), \verb"ind"(\bar{\partial}^{t,\theta}:W^{p,B}_{1,1+\kappa_{j,l}}\rightarrow W^{p}_{0,1+\kappa_{j,l}}), l=1,\cdots, k.$
[10] has got index decomposition theorems which relate local index problems above to $\verb"ind"(\bar{\partial}^{t,\theta}:\hat{W}^{p}_{1,1+\kappa}\rightarrow W^{p}_{0,1+\kappa})$, $\verb"ind"(\bar{\partial}^{t,\theta}:W^{p}_{1,1+\kappa}\rightarrow W^{p}_{0,1+\kappa})$ as follows:
\begin{theorem}([10])
In case of the smooth metric, we have \\
\\
$1)\verb"ind"(\bar{\partial}^{t,\theta}:\hat{W}^{p}_{1,1+\kappa}\rightarrow W^{p}_{0,1+\kappa})
\\=\verb"ind"(\bar{\partial}^{t,\theta}:W^{p,B}_{1}(inn)\rightarrow W^{p}_{0}(inn))+\sum^{k}_{l=1}\verb"ind"(\bar{\partial}^{t,\theta}:\hat{W}^{p,B}_{1,1+\kappa_{j,l}}\rightarrow W^{p}_{0,1+\kappa_{j,l}})$\\
\\
$2)\verb"ind"(\bar{\partial}^{t,\theta}:W^{p}_{1,1+\kappa}\rightarrow W^{p}_{0,1+\kappa})\\=\verb"ind"(\bar{\partial}^{t,\theta}:W^{p,B}_{1}(inn)\rightarrow W^{p}_{0}(inn))+\sum^{k}_{l=1}\verb"ind"(\bar{\partial}^{t,\theta}:W^{p,B}_{1,1+\kappa_{j,l}}\rightarrow W^{p}_{0,1+\kappa_{j,l}})$
\end{theorem}
This index decomposition theorem is our start of later index computation. By this theorem, to compute $\verb"ind"(\bar{\partial}^{t,\theta}: W_{1,1+\kappa}^{p}\rightarrow W_{0,1+\kappa}^{p})$ is equivalent to compute
 \begin{eqnarray*}\verb"ind"(\bar{\partial}^{t,\theta}:W^{p,B}_{1}(inn)\rightarrow W^{p}_{0}(inn)) \end{eqnarray*}
 and
 \begin{eqnarray*}\verb"ind"(W^{p,B}_{1,1+\kappa_{j,l}}\rightarrow W^{p}_{0,1+\kappa_{j,l}})\end{eqnarray*}.\\
$\textbf{Remark 2.1.}$  This theorem is in case of the smooth metric. In section 3, we will generalize this theorem in case of the cylindrical metric.

\subsection{Riemann-Roch theorem with boundary}
\hspace{14pt} First let's focus on the computation problem of \begin{eqnarray*} \verb"ind"(\bar{\partial}^{t,\theta}:W^{p,B}_{1}(inn)\rightarrow W^{p}_{0}(inn)) \end{eqnarray*}
McDuff and Salamon ([14]) have already studied similar problems and got the Riemann-Roch theorem with boundary.
As an application of their work, we will compute
\begin{eqnarray*} \verb"ind"(\bar{\partial}^{t,\theta}:W^{p,B}_{1}(inn)\rightarrow W^{p}_{0}(inn)) \end{eqnarray*}
\par
Note that orbifold structure of $L_{j}$ near marked point $z_{l}$ is given by: \begin{eqnarray}
e^{\frac{2\pi i}{m}}(z,m)=(e^{\frac{2\pi i}{m}}z,e^{2\pi i(a_{j}(h_{l}))}w)=(e^{\frac{2\pi i}{m}}z,e^{2\pi i(\frac{v_{j,l}}{m})}w)  \end{eqnarray}
According to the definition of $L_{j}$ , there is a natural boundary condition for $|L_{j}|$, that is: \begin{equation}
\Psi_{l}(e^{i\theta})\mathbb{R}=e^{iv_{j,l}\theta}\mathbb{R}  \end{equation} \\
$\textbf{Remark 2.2.}$ Note that we delete the factor $2\pi$ in above equality for complying with the framework of the appendix C in [14], but it will not influence our later computation.\par
Now we can apply [14]'s Riemann-Roch theorem  with boundary to our problem. We cite their theorem (Theorem C.1.10 of [14]) as follows:\\
\begin{theorem}[Riemann-Roch theorem with boundary]
Let $E\rightarrow \Sigma$ be a complex vector bundle on a compact Riemannian surface with boundary and $F\subset E|_{\partial\Sigma}$ be a real subbundle. Let $D$ be a real $Cauchy-Riemann$ operator on $E$ of class $W^{l-1,p}$, where $l$ is a positive integer and $p>1$ such  that $lp>2$. Then the following holds for every integer $k\in\{1,2,\cdots,l\}$ and every real number $q>1$ such that $k-\frac{2}{q}\leq l-\frac{2}{p}$.\\
(1)The operators \begin{eqnarray*}
D_{F}:W^{k,q}_{F}(\Sigma,E)\rightarrow W^{k-1,q}(\Sigma,\wedge^{0,1}T^{*}\Sigma\otimes E),\\
D_{F}^{*}:W^{k,q}_{F}(\Sigma,\wedge^{0,1}T^{*}\Sigma\otimes E)\rightarrow W^{k-1,q}(\Sigma,E) \end{eqnarray*}
are Fredholm operators.\\
(2)The real Fredholm index of $D_{F}$ is given by \begin{equation}
\verb"ind"(D_{F})=n\chi(\Sigma)+\mu(E,F)  \end{equation}
where $\chi(\Sigma)$ is the Euler chracteristic of $\Sigma$, $\mu(E,F)$ is the boundary Maslov index (see the Appendix), n is the complex rank of $E$.\par
\end{theorem}
Note that $\Sigma \setminus \cup_{l=1}^{k}B_{1}(z_{l})$ is a Riemannian surface of genus g with boundary, $E=|L_{j}||_{\Sigma\setminus\cup_{l=1}^{k}B_{1}(z_{l})}$ is a complex vector bundle on $\Sigma\setminus\cup_{l=1}^{k}B_{1}(z_{l})$, $F=|L_{j}||_{\cup_{l=1}^{k}\partial (B_{1}(z_{l}))}$ is a real subbundle of $E$, and $\partial^{t,\theta}$ is a real Cauchy-Riemann operator on $E$.
By Riemann-Roch theorem with boundary above, we get\begin{eqnarray}
\verb"ind"(W^{p,B}_{1}(inn)\rightarrow W^{p}_{0}(inn))=\chi(\Sigma)+\mu(E,F)=(2-2g-k)+\mu(E,F)  \end{eqnarray} \par
Therefore, it suffices to compute the boundary Maslov index $\mu(E,F)$. \par
Let $\Sigma_{01}=\Sigma \setminus \cup_{l=1}^{k}B_{1}(z_{l})$, $\Sigma_{12}$ disjoint union of $k$ closed unit discs, $E_{01}=E, F_{01}=F$, $E_{12}, F_{12}$  bundles on $\Sigma_{12}$ with the same boundary conditions as $E, F$ respectively, $\Sigma_{02}$ closed Riemannian surface corresponding to $\Sigma_{01}$, $E_{02}$ extension of $E$ on $\Sigma_{02}$, $F_{02}=\emptyset$.\par
According to the definition of decomposition (see Definition 4.1, 4.2 in the Appendix), it is easy to check that ($\Sigma_{01},\Sigma_{12})$ is a decomposition of $\Sigma_{02}$, $(E_{01},F_{01})$, $(E_{12},F_{12})$ are a bundle pair decomposition of $(E_{02},F_{02})$.
So, by the composition axiom in Theorem 4.1, we have:
\begin{eqnarray}
\mu(E,F)=\mu(E_{02},\emptyset)- \mu(E_{12},F_{12}) \end{eqnarray}
By Theorem 4.2, \begin{eqnarray}
 \mu(E_{02},\emptyset)=2 c_{1}(E_{02})[\Sigma_{02}]=2 c_{1}(|L_{j}|)([\Sigma_{02}])=2deg(|L_{j}|) \end{eqnarray} \par
Therefore the original index problem can be translated to be a degree computation of the line bundle $|L_{j}|$. This has been done in [11].\\
\begin{theorem}([11]) We can compute the degree of $|L_{j}|$ as follows
\begin{eqnarray}
deg(|L_{j}|=q_{j}(2g-2+k)-\sum_{l=1}^{k} a_{j}(h_{l})\in \mathbb{Z} \end{eqnarray}\par
\end{theorem}
Therefore, by (18),(19) we get \begin{equation}
\mu(E_{02},\emptyset)=2q_{j}(2g-2+k)-2\sum_{l=1}^{k} a_{j}(h_{l})\in \mathbb{Z} \end{equation} \par
The remaining problem is to compute the Maslov index $\mu(E_{12},F_{12})$.
It is easy considering our given boundary conditions (14), by Theorem 4.1, we get
\begin{equation}\mu(E_{12},F_{12})= -2\sum_{l=1}^{k} v_{j,l} \end{equation}
By (17), (20), (21), we get  \begin{equation} \mu(E,F)=2q_{j}(2g-2+k)-2\sum_{l=1}^{k} a_{j}(h_{l})+2\sum_{l=1}^{k} v_{j,l} \end{equation}\par
Combining Theorem 2.3, (16) and (22), we get
\begin{theorem} \begin{eqnarray*}
&&\verb"ind"(\bar{\partial}^{t,\theta}:W^{p,B}_{1}(inn)\rightarrow W^{p}_{0}(inn))\\
&=&(1-2q_{j})(2-2g-k)-2\sum_{l=1}^{k} a_{j}(h_{l})+2\sum_{l=1}^{k} v_{j,l}\in\mathbb{ Z} \end{eqnarray*}
\end{theorem}
\subsection{$L^{p}$-index gluing theorem}
\hspace{14pt} Now let us compute the index  \begin{eqnarray*} \verb"ind"(\bar{\partial}^{t,\theta}:{W^{p,B}_{1,1+\kappa_{j,l}}}\rightarrow W^{p}_{0,1+\kappa_{j,l}}) \end{eqnarray*}
Note that here the base manifold is an end $S^{1}\times [0,\infty)$, so this is an index problem on noncompact manifold on which classical Atiyah-Singer theorems ([1-7]) do not work. However, this kind of manifold is the easiest case of noncompact manifold, whose index theory has been studied by Donaldson, Lockhart, McOwen ([9],[13]), etc. We need apply and generalize their results to solve our problems. \par
In [9], Donaldson studied the index theory over a tubular manifold and got index gluing theorems (see the Appendix). In fact, he mainly considered the 4-dimension case, but his theorems also can fit for general case. Here as an application of his work, we will prove the $L^{p}$-index gluing theorem with weights for 2-dimension case.\par
First consider the following gluing problem with only two ends. Assume that Riemannian surface $\Sigma$ is a disjoint union of two disconnected components $\Sigma=\Sigma_{1}\cup\Sigma_{2}$, and two ends $ Y\times (0,\infty),\bar{Y}\times (0,\infty)$ whose orientations are oppostie are contained in different components, where $Y=S^{1}$. Suppose there are vector bundles $E_{1},E_{2}$ on $\Sigma_{1},\Sigma_{2}$ respectively. Then we can define Sobolev spaces $L^{p}(E_{i}), L^{p}_{1}(E_{i})$ and differential operators
 \begin{eqnarray*}D_{i}: L^{p}_{1}(E_{i})\rightarrow L^{p}(E_{i}) \end{eqnarray*}
Assume that $D_{i}$ can be written as $D_{i}=\frac{d}{dt}+L_{i},i=1,2$, where $L_{i}$ are self-dual elliptic operators. \par
Now we consider the Riemannian surfaces $\Sigma^{\sharp}$ obtained by identifying the two ends of $\Sigma$, then construct $E^{\sharp}$ over $\Sigma^{\sharp}$ and the operator $D^{\sharp}: L^{p}_{1}(E^{\sharp})\rightarrow L^{p}(E^{\sharp})$ (see [9] or the Appendix for more details).\par
We can also prove these operators are Fredholm operators, then define their Fredholm indices $\verb"ind"(D_{i}),\verb"ind"(D^{\sharp})$ as [9]. Moreover we can prove \par
\begin{theorem}[$L^{p}$-Index gluing theorem]
In situations above, assume that operators $L_{1}, L_{2}$ are invertible in the decomposition $D_{1}=\frac{d}{dt}+L_{1},D_{2}=\frac{d}{dt}+L_{2}$, we have
\begin{equation}\verb"ind"(D^{\sharp})=\verb"ind"(D_{1})+\verb"ind"(D_{2}) \end{equation}
\end{theorem}
$Proof$: See the Appendix. $\Box$. \\
\\
$\textbf{Remark 2.3.}$ The key point of the proof is that differential operators $D_{i}$ can be decomposed as $D_{i}=\frac{d}{dt}+L_{i}$ as Donaldson did. Therefore, it is similar to his proof of $L^{2}$-edition.\par
 When the operators $L_{i}(i=1,2)$ are not invertible, we must introduce weights $\alpha_{i}\in\mathbb{R}$ and consider weighted Sobolev spaces
$L^{p,\alpha_{i}}_{1}(E_{i}), L^{p,\alpha_{i}}(E_{i})$, and
\begin{eqnarray*}
D^{\alpha_{i}}: L^{p,\alpha_{i}}_{1}(E_{i})\rightarrow L^{p,\alpha_{i}}(E_{i}), i=1,2
\end{eqnarray*}
When $\alpha_{2}=-\alpha_{1}=-\alpha$, we can similarly glue the two ends and get
\begin{eqnarray*}D^{\sharp}: L^{p,\alpha}_{1}(E^{\sharp})\rightarrow L^{p,\alpha}(E^{\sharp}) \end{eqnarray*}
However, the introduction of weighted Sobolev space is equivalent to replace the operator $L$ of $(L-\alpha)$ in Sobolev space without weights.
Therefore,we can easily generalize the index gluing theorem above to the case with weights as follows (see [9]).
\begin{theorem}[$L^{p}$-Index gluing theorem with weights I]
Assuming $\alpha\in \mathbb{R} $ such that $L_{i}-\alpha(i=1,2)$ is invertible, we have \begin{eqnarray}
\verb"ind"(D^{\sharp})=\verb"ind"(D^{\alpha})+\verb"ind"(D^{-\alpha}) \end{eqnarray}
\end{theorem}
We can further consider weight vector case which corresponds to more ends.
Choose a weight $\alpha_{i}$ for each end $Y_{i}\times (0,\infty)$ of $\Sigma$ and define a weight vector $\vec{\alpha}=(\alpha_{1},\cdots,\alpha_{N})$. Fix a positive function $W$ on $\Sigma$ which is equal to $e^{\alpha_{i}t}$ on the $i$th end and define norms\begin{eqnarray*}
   ||f||_{L^{p,\vec{\alpha}}}=||Wf||_{L^{p}},||f||_{L_{1}^{p,\vec{\alpha}}}=||Wf||_{L_{1}^{p}} \end{eqnarray*}
then we can define $L^{p,\vec{\alpha}},L_{1}^{p,\vec{\alpha}}$ and
\begin{eqnarray*} D^{\vec{\alpha}}:L_{1}^{p,\vec{\alpha}}\rightarrow L^{p,\vec{\alpha}}\end{eqnarray*}
Similar to [9], we can easily obtain the index gluing theorem in weight vector case.
\begin{theorem}[$L^{p}$-Index gluing theorem with weights II]
\begin{eqnarray}
\verb"ind"(D^{\sharp;(\alpha_{2},\cdots,\alpha_{N})})=\verb"ind"(D^{(\alpha_{1},\alpha_{2},\cdots,\alpha_{N})})
+\verb"ind"(D^{(-\alpha_{1},\alpha_{2},\cdots,\alpha_{N})})\end{eqnarray}
\end{theorem}
\subsection{Lockhart-McOwen theory}
\hspace{14pt} In this part, we recollect the work of Lockhart and McOwen ([13]) for general elliptic operators defined on a noncompact manifolds with finite ends. In next section we will compute the index using their work.\par
Suppose $X$ is an $n$-dimensional noncompact manifold without boundary, containing a compact set $X_{0}$ such that
\begin{eqnarray*} X\setminus X_{0}=\{(\omega,\tau):\omega\in \Omega, \tau\in(0,\infty)\}\end{eqnarray*}
where $\Omega$ is a $n-1$-dimensional closed Riemannian manifold with a smooth measure $d\omega$.\par
Let $E,F$ be rank-d vector bundles over $X$. Denote by $C^{\infty}(E)$ the set of smooth sections and $C_{0}^{\infty}(E)$ the set of smooth sections with compact supported sets. Choose a finite cover $\{\Omega_{1},\cdots,\Omega_{N}\}$ of coordinate patches of $\Omega$ and let $X_{v}=\Omega_{v}\times (0,+\infty)$. We can continue to choose a covering $X_{N+1},\cdots, X_{M}$ of coordinate patches of $X_{0}$ such that $E$ can be trivialized over $X_{v},v=1,\cdots,N,\cdots,M$. Let $u=(u_{1},\cdots,u_{d})$ be a trivialization of a section $u$ with compact supported set over $X_{v}$, we can define the norm
\begin{eqnarray*} ||u||_{W^{p}_{s}(X_{v})}:=\sum_{|\alpha|\leq s}\sum_{l=1}^{d}||D^{\alpha}u_{l}||_{W^{p}_{0}(X_{v})}, (D=-i\partial/\partial x)
\end{eqnarray*}
where we use the measure $d\omega d\tau$ if $v=1, \cdots, N$. Let $\varphi_{1},\cdots, \varphi_{N+M}$ be a set of $C^{\infty}$ partition functions subordinate to the cover $X_{1},\cdots, X_{N+M}$. We define a norm on $C_{0}^{\infty}(E)$ by
\begin{eqnarray*} ||u||_{W_{s}^{p}}:=\sum_{v=1}^{N+M}||\varphi_{v}u||_{W_{s}^{p}(X_{v})} \end{eqnarray*}
and let $W_{s}^{p}(E)$ be the closure of $C_{0}^{\infty}(E)$ in this norm. We can add a weight at infinity to generalize this space. Over $X_{v}, v=1,\cdots, N$ we define the weighted norm
\begin{eqnarray*}||u||_{W_{s,k}^{p}(X_{v})}:=\sum_{|\alpha|\leq s}\sum_{v=1}^{d}||e^{k\tau}D^{\alpha}u_{l}||_{W_{0}^{p}(X_{v})} \end{eqnarray*}
and replace $W_{s}^{p}(E)$ by $W_{s,k}^{p}(E)$ whose norm is given below
\begin{eqnarray*}||u||_{W_{s,k}^{p}}:=\sum_{v=1}^{N}||\varphi_{v}u||_{W_{s,k}^{p}(X_{v})}+\sum_{v=N+1}^{N+M}||\varphi_{v}u||_{W_{s}^{p}(X_{v})} \end{eqnarray*}
Similarly we can define $W_{r,k}^{p}(F)$, where $s=(s_{1},\cdots,s_{I}), r=(r_{1},\cdots,r_{J})$ are multiple indices.
Suppose $A: C_{0}^{\infty}(E)\rightarrow C_{0}^{\infty}(F)$ is translation invariant elliptic operator with respect to $(s,r)$.
Then $A: W_{s,k}^{p}(E)\rightarrow W_{r,k}^{p}(F)$ is a bounded operator. Furthermore, Lockhart and McOwen([13]) proved the following theorem
\begin{theorem}[Index jumping formula]
Suppose $A$ is elliptic with respect to $(r,s)$ and is translation invariant when $\tau >0$. Then we have:\\
(1)There exists a discrete subset $\mathfrak{D}_{A} \subset \mathbb{R}$ such that the operator:\\
$A:W^{p}_{s,k}(E)\rightarrow W^{p}_{r,k}(F)$ is Fredholm operator if and only if $k\in \mathbb{R}\setminus \mathfrak{D}_{A}$.\\
(2)For $k_{1},k_{2}\in \mathbb{R}\setminus \mathfrak{D}_{A}$ with $k_{1}<k_{2}$, there is
\begin{eqnarray} i_{k_{1}}(A)-i_{k_{2}}(A)=N(k_{1},k_{2}) \end{eqnarray}
where $i_{k_{j}}$ is the Fredholm index of $A: W^{p}_{s,k_{j}}(E)\rightarrow W^{p}_{r,k_{j}}(F), j=1,2$,\begin{eqnarray}
N(k_{1},k_{2})=\Sigma\{d(\lambda):\lambda\in\mathfrak{E}_{A},k_{1}<Im(\lambda)<k_{2}\}  \end{eqnarray}
where $\mathfrak{E}_{A}$ is the spectrum of $A$, $\mathfrak{D}_{A}:=\{Im(\lambda)\in \mathbb{R}: \lambda\in \mathfrak{E}_{A}\}$, $d(\lambda)$ is the dimension of the eigenspace corresponding to the spectrum point $\lambda$.
\end{theorem}

\subsection{Index computation}
$\textbf{Proof of Theorem 1.3}$: First we compute
\begin{eqnarray*}\verb"ind"(\bar{\partial}^{t,\theta}: W^{p,B}_{1,1+\kappa_{j,l}}\rightarrow W^{p}_{0,1+\kappa_{j,l}}) \end{eqnarray*}
where $\kappa_{j,l}=-a_{j}(h_{l})+q_{j}-2/p$.\par
 We have a decomposition $\bar{\partial}=\partial_{t}+i\partial_{\theta}=\frac{d}{dt}+L$.
 Note that $L=i\partial_{\theta}$ is not invertible (Because its spectrum is $\mathbb{Z}$, and ker$L\cong \mathbb{C}$), so we need introduce weights as discussion above .\par
Naturally consider the number $1+\kappa_{j,l}$ as a weight: $\alpha=\alpha^{+}=1+\kappa_{j,l}$, and define $\alpha^{-}=-\alpha^{+}$.
Write ${\bar \partial}^{\alpha^{+}}, {\bar \partial}^{\alpha^{-}}$ as ${\bar \partial}^{+}, {\bar \partial}^{-}$ respectively. Then we have $\verb"ind"(\bar \partial)=\verb"ind"({\bar \partial}^{+})$. By $L^{p}$-index gluing theorem with weights I (Theorem 2.6), we get
\begin{eqnarray}
\verb"ind"({\bar \partial}^{\sharp})&=&\verb"ind"({\bar \partial}^{+})+\verb"ind"({\bar \partial}^{-})\end{eqnarray}
where the operator ${\bar \partial}^{\sharp}$ is defined over the compact Riemann surface $S^{1}\times I$, the boundary conditions only need reverse (14), \begin{equation}
\Psi_{l}(e^{i\theta})\mathbb{R}=e^{-iv_{j,l}\theta}\mathbb{R}  \end{equation}
 By Theorem 2.2, we get :
\begin{eqnarray}
\verb"ind"({\bar \partial}^{\sharp})=\chi(S^{1}\times I)+\mu(E,F)=\mu(E,F)\end{eqnarray}
Therefore, the computation of $\verb"ind"({\bar \partial}^{\sharp})$ can be transformed to be computation of boundary Maslov index $\mu(E,F)$. \par We adopt previous methods. First consider a disjoint union of two discs with opposite boundary conditions as $S^{1}\times I$, then glue the two discs on $S^{1}\times I$ along the boundary. So by Theorem 2.2, Theorem 4.1 and Theorem 4.2, we get :
\begin{eqnarray}
\verb"ind"({\bar \partial}^{\sharp})=\mu(E,F)=4-4v_{jl} \end{eqnarray}\\
On the other hand, by our index jumping formula (Theorem 2.8) we get
\begin{equation}
\verb"ind"({\bar \partial}^{-})-\verb"ind"({\bar \partial}^{+})=dim_{\mathbb{R}}\mathbb{C}=2. \end{equation}
Therefore by (28), (31), (32), we get
\begin{theorem}\begin{eqnarray}
\verb"ind"(\bar{\partial}^{t,\theta}:W^{p,B}_{1,1+\kappa_{j,l}}\rightarrow W^{p}_{0,1+\kappa_{j,l}})
  =1-2v_{jl}\end{eqnarray}
\end{theorem}
Combining Theorem 1.2, Theorem 2.1, Theorem 2.4, and Theorem 2.9, we ultimately get
\begin{eqnarray*}
&&\verb"ind"(\bar{\partial}: L_{1}^{p}(\Sigma, |L_{j}|)\rightarrow L^{p}(\Sigma, |L_{j}|\otimes \wedge^{0,1}))\\
&=&k+(1-2q_{j})(2-2g-k)-2\sum_{l=1}^{k}a_{j}(h_{l})+\sharp\{z_{l}:c_{jl}<0\}
\end{eqnarray*}
where $c_{jl}=a_{j}(h_{l})-q_{j}$, $q_{j}$ is the fractional degree of a nondegenerate quasi-homogeneous polynomial $W$ with respect to the jth variable, $a_{j}(h_{l})=v_{j,l}/m_{l}$  is the orbifold action on the line bundle $L_{j}$ at the marked point $z_{l}$. \par
And the right side is actually an integer, i.e.
\begin{eqnarray*}k+(1-2q_{j})(2-2g-k)-2\sum_{l=1}^{k} a_{j}(h_{l})
+\sharp\{z_{l}:c_{jl}<0\}\in \mathbb{Z} \end{eqnarray*}
This completes the proof of Theorem 1.3. $\Box$. \\
\\
$\textbf{Remark 2.4.}$ If we write index formula above as \begin{eqnarray*}
&&\verb"ind"(\bar{\partial}: L_{1}^{p}(\Sigma, |L_{j}|)\rightarrow L^{p}(\Sigma, |L_{j}|\otimes \wedge^{0,1}))\\
&=&2(1-2q_{j})(1-g)-2\sum_{l=1}^{k}(\Theta_{j}^{\gamma_{l}}-q_{j})+\sharp\{z_{l}:c_{jl}<0\}
\end{eqnarray*}
where $\Theta_{j}^{\gamma_{l}}=a_{j}(h_{l})$. We can see that this index formula is almost the same as the index formula ([12]) in case of the cylindrical metric except for the term $\sharp\{z_{l}:c_{jl}<0\}$.\\
\section{Index computation in case of the cylindrical metric}

\hspace{14pt} In [12], they introduced the linearized operator $D$ of Witten map, that is:\begin{eqnarray*}
D=D_{\wp,\mu}WI:{L_{1}}^{p}(\Sigma,L_{1}\times L_{2}\times \cdots \times L_{N})\rightarrow L^{p}(\Sigma,L_{1}\otimes\wedge^{0,1})\times\cdots L^{p}(\Sigma,L_{N}\otimes\wedge^{0,1}) \end{eqnarray*}
Then they proved this operator is a Fredholm operator under some mild conditions in case of the cylindrical metric and computed its index.
The next interesting question is:
If we add weight $\delta$ to Sobolev spaces above, and consider the following operator:\begin{eqnarray*}
D^{\delta}=D_{\wp,\mu}WI:{L_{1}}^{p,\delta}(\Sigma,L_{1}\times L_{2}\times \cdots \times L_{N})\rightarrow L^{p,\delta}(\Sigma,L_{1}\otimes\wedge^{0,1})\times\cdots L^{p,\delta}(\Sigma,L_{N}\otimes\wedge^{0,1}) \end{eqnarray*}
then what is the  relation between $\verb"ind"(D^{\delta})$ and $\delta$ ?\par
For convenience, let's first consider the operator ${\bar \partial}_{j}^{\delta}$ defined in weighted Sobolev space \begin{eqnarray*}
{\bar \partial}_{j}^{\delta}:{L_{1}}^{p,\delta}(\Sigma,|L_{j}|)\rightarrow L^{p,\delta}(\Sigma,|L_{j}|\otimes \wedge^{0,1}), j=1,\cdots,N \end{eqnarray*}\par
where $|L_{j}|$ is the desingularization of orbifold line bundle $L_{j}$. \par

Let $(\Sigma,z_{1},\cdots,z_{k})$ be an orbicurve with k marked points, $B_{1}(z_{l})$ unit closed disc with the center $z_{l}$. Choose a compact subset $\Omega\subset\Sigma\setminus\cup B_{e^{-1}}(z_{l})$ such that $\Sigma,B_{1}(z_{1}),\cdots,B_{1}(z_{k})$ can cover $\Sigma$. Let $\varphi_{0},\cdots,\varphi_{k}$ be a set of partition functions subordinate to the cover. Let $e_{j}$ be basis of orbifold line bundle $L_{j}$ on $\Sigma$, and recall the cylindrical metric is defined as: $|e_{j}|=|z|^{a_{j}(z_{l})}$. Let section of $L_{j}$ on $B_{1}(z_{1})$ be $u_{j}=\tilde{u}_{j}e_{j}$, we can define norm $||\cdot||_{p}, ||\cdot||_{1,p}$, $L_{1}^{p}(\Sigma,|L_{j}|)$, and the operator $\bar \partial:L_{1}^{p}(\Sigma,|L_{j}|)\rightarrow L^{p}(\Sigma,|L_{j}|\otimes\wedge^{0,1})$ almost the same as the smooth metric case (see section 1).\par
Likewise, we also make a coordinate transformation $z=e^{-t-i\theta}$ and change $\bar \partial$ as ${\bar \partial}^{t,\theta}$, and their relation is as follows:\begin{eqnarray*}
\bar \partial=-\frac{1}{2}e^{t-i\theta}(\partial_{t}+i\partial_{\theta})=-e^{t-i\theta}{\bar \partial}^{t,\theta} \end{eqnarray*}
We write ${\bar \partial}_{j}^{\delta}$ as ${\bar \partial}^{\delta}$ if no confusion occurs. So \begin{eqnarray*}
{\bar \partial}^{\delta}:L_{1}^{p,\delta}(\Sigma,|L_{j}|)\rightarrow L^{p,\delta}(\Sigma,|L_{j}|\otimes\wedge^{0,1}) \end{eqnarray*}
is transformed as follows:
\begin{eqnarray*}{\bar \partial}^{t,\theta,\delta}:\hat{W}_{1,1+k+\delta}^{p}\rightarrow W_{0,1+k+\delta}^{p} \end{eqnarray*}
where norms $\hat{W}_{1,1+k+\delta}^{p}, W_{0,1+k+\delta}^{p}$ are defined as :\begin{eqnarray*}
||u_{j}||_{\hat{W}_{1,1+k+\delta}^{p}}&=&(\int_{S^{1}\times [0,\infty)}|\tilde{u}_{j}|^{p}e^{(k+\delta)pt}+(|\partial\tilde{u}_{j}|^{p}+|\bar \partial\tilde{u}_{j}|^{p})e^{(1+k+\delta)pt})^{1/p} \\||u_{j}||_{{W}_{0,1+k+\delta}^{p}}&=&(\int_{S^{1}\times [0,\infty)}|\tilde{u}_{j}|^{p}e^{(1+k+\delta)pt})^{1/p} \end{eqnarray*}

\subsection{Index transformation theorem}
\hspace{14pt}Unfortunately, the space $\hat{W}_{1,1+k+\delta}^{p}$ is not a normal weighted Sobolev space (The normal one is $ W_{1,1+k+\delta}^{p}$), so we can not directly apply Lockhart-McOwen theory above. Therefore, first we should transform this problem into a normal case, which needs generalize the index transformation theorem (Theorem 1.2) to the case of the cylindrical metric. That is, we want to prove
\begin{theorem}(Index transformation theorem) In case of the cylindrical metric, if $1<p<\frac{2}{q_{j}}$ and $a_{j}(h_{l})+\frac{2}{p}\neq1,2$ for any $l (l=1,\cdots,k)$, then $\bar{\partial}:L_{1}^{p}(\Sigma, |L_{j}|)\rightarrow L^{p}(\Sigma, |L_{j}|\otimes \wedge^{0,1})$ is a Fredholm operator.\par
In particular, if $p>2$ we have the index relation  \begin{eqnarray*}
&&\verb"ind"(\bar{\partial}: L_{1}^{p}(\Sigma, |L_{j}|)\rightarrow L^{p}(\Sigma, |L_{j}|\otimes \wedge^{0,1}))\\&=&\verb"ind"(\bar{\partial}^{t,\theta}: W_{1,1+\kappa}^{p}\rightarrow W_{0,1+\kappa}^{p})+\sharp\{z_{l}: c_{jl}=0\} \end{eqnarray*}
and the index is independent of $p$ in the interval $(2, \infty)$.\end{theorem}
Next we will prove this theorem step by step as [10].\\
\\
$\textbf{Remark 3.1.}$  The parameter $c_{ij}$ is very important, where $c_{jl}=a_{j}(h_{l})-q_{j}$ in case of the smooth metric, and
$c_{jl}=a_{j}(h_{l})$ in case of the cylindrical metric. It is obvious that both cases $c_{jl}\geq0, c_{jl}<0$ can happen in case of the smooth metric, but only the case $c_{jl}\geq0$ can happen in case of the cylindrical metric. \par
We still use the same notations as [10]. Firstly we can obtain local estimate of special solution in case of the cylindrical metric similar to Lemma 4.3 in [10]:
\begin{lemma}
If $f\in L^{p}(B_{1}(0),|L_{j}|\otimes\wedge^{0,1})$ for $p$ satisfying the condition $a_{j,l}=a_{j}(h_{l})+2/p\in \mathbb{R}\setminus \mathbb{Z}$, then the special solution $u_{s}=Q_{s}\circ f$ satisfies the following estimates:\\
(1)if $1<p<\infty$, then \begin{equation}
||u_{s}||_{1,p;B_{1}(0)}+||\frac{u_{s}}{z}||_{p;B_{1}(0)}\leq C||f||_{p;B_{1}(0)} \end{equation}
(2)if $1<p\leq2$, and $1<q<\frac{2p}{2-p}$,then\begin{equation}
||u_{s}||_{q;B_{1}(0)}\leq C||u_{s}||_{1,p;B_{1}(0)}\leq C||f||_{p;B_{1}(0)};\end{equation}
(3)if $p>2$, and $0<\alpha<1-\frac{2}{p}$, then \begin{equation}
||\tilde{u}_{s}r^{c}||_{C^{\alpha}(B_{1}(0)}\leq C||u_{s}||_{1,p;B_{1}(0)}\leq C||f||_{p;B_{1}(0)} \end{equation}
where $c=a_{j}(h_{l})$.
\end{lemma}
$Proof$: Through serious check on Lemma 4.3 in [10], we can find the whole proof views the constant $c=a_{j}(h_{l})-q_{j}$ as a unity. So the proof does not change if we replace $c=a_{j}(h_{l})-q_{j}$ as $c=a_{j}(h_{l})$ in case of cylindrical metric.                       $\Box$ \par
Next we consider estimate of the homogeneous solution in case of the cylindrical metric. Compared to estimate of the homogeneous solution in case of the smooth metric, we may find it is easier to be dealed with because there is only one case (see Remark 3.1). We have the following lemma similar to Lemma 4.4 in [10].
\begin{lemma}
Let $\bar \partial u=0$ and  $u\in L^{p}(B_{1}^{+}(0),|L_{j}|)$ for $p>1$. We have the estimate:\\
(1)for any $k\geq0$ and $1<q<\infty$, there exists a $C$ such that \begin{equation}
||\tilde{u}||_{W^{q}_{k}(B_{1}(0)}\leq C||u||_{p;B^{+}_{1}(0)\setminus B_{\frac{1}{2}}(0)} \end{equation}
(2)if $c\geq0$, then for $1<q<\infty$, there exists a $C$ such that \begin{equation}
||u||_{1,q;B_{1}(0)}\leq C||u||_{p;B^{+}_{1}(0)\setminus B_{\frac{1}{2}}(0)} \end{equation}
(3)if $c<0$,then for $1<q<\frac{2}{q_{j}}$, there exists a constant $C$ such that the above inequality in (2) holds.
\end{lemma}
$Proof$: Because the original proof of Lemma 4.4 in [10] viewed the constant $c$ as a unity again, so it is obvious.  $\Box$ \\
\par
Combining Lemma 3.2 and Lemma 3.3, we have theorems in case of the cylindrical metric corresponding to Corollary 4.5 and Lemma 4.6 in [10]:
\begin{lemma} If $c>0$ at $z_{l}=0$, then for $1<p<2/(1-{\bar \delta}_{j})$, where ${\bar \delta}_{j}=min_{l:c_{jl}>0}(c_{jl})$, there is \begin{eqnarray*}
||u||_{1,p;B_{1}(0)}+||\frac{u}{z}||_{p;B_{1}(0)}\leq C||u||_{1,p;B_{1}(0)} \end{eqnarray*}
\end{lemma}
$Proof$: Use the same argument in the proof above again.$\Box$ \\
\par
For $c=a_{j}(h_{l})=0$, we have similar estimates (Note that $c=a_{j}(h_{l})\leq0$ in Lemma 4.6 of [10])
\begin{lemma}  If $c=a_{j}(z_{l})=0$,then for $1<p<\infty$ and any $u=\tilde{u}e_{j}\in L^{p}_{1}(B_{1}(0),|L_{j}|)$ satisfying $u(0)=0$,there is \begin{equation} ||u||_{1,p;B_{1}(0)}+||\frac{u}{z}||_{p;B_{1}(0)}\leq C||u||_{1,p;B_{1}(0)} \end{equation}
\end{lemma}
$Proof$: It is the same as the proof of Lemma 4.6 in [10]. $\Box$ \\
\par
Next we have regularity of local solution.
\begin{lemma} Let $\bar \partial u=f$ in $B^{+}_{1}(0)$, where $u\in L^{p}_{1}(B^{+}_{1}(0),|L_{j}|)$ and $f\in L^{p}_{1}(B^{+}_{1}(0),|L_{j}|\otimes\wedge^{0,1})$,
then $u\in L^{p}_{1}(B^{+}_{1}(0),|L_{j}|)$ and the inequality \begin{equation}
||u||_{1,p;B_{1}(0)}\leq C(||u||_{p;B^{+}_{1}(0)\setminus B_{\frac{1}{2}}(0)}+||f||_{p;B^{+}_{1}(0)}) \end{equation}
holds if the following two conditions are satisfied:\par
$\bullet a_{j,l}=a_{j}(h_{l})-q_{j}+2/p \in \mathbb{R}\setminus\mathbb{Z}$\par
$\bullet c\geq0,1<p<\infty$
\end{lemma}
$Proof$: The proof is similar to the proof of Lemma 4.7 in [10]. Under the assumptions on parameters $c$ and $p$, one has \begin{eqnarray*}||u||_{1,p;B_{1}(0)}&\leq&||u-u_{s}||_{1,p;B_{1}(0)}+||u_{s}||_{1,p}\\
&\leq& C(||u-u_{s}||_{p;B^{+}_{1}(0)\setminus B_{\frac{1}{2}}(0)}+||f||_{p})\\
&\leq& C(||u||_{p;B^{+}_{1}(0)\setminus B_{\frac{1}{2}}(0)}+||u_{s}||_{p}+||f||_{p})\\
&\leq& C(||u||_{p;B^{+}_{1}(0)\setminus B_{\frac{1}{2}}(0)}+||f||_{p;B^{+}_{1}(0)})\\
\end{eqnarray*}
where the second inequality comes from (2) of  Lemma 3.2 and (1) of Lemma 3.3,
and the fourth inequality comes from Lemma 3.2. $\Box$ \par
Now by the above lemma, we can obtain the following global estimate.
\begin{lemma}
Let $\bar \partial u=f$ on $\Sigma$, where $u\in L^{p}_{1}(\Sigma,|L_{j}|)$ and $f\in L^{p}_{1}(\Sigma,|L_{j}|\otimes\wedge^{0,1})$. Then $u\in L^{p}_{1}(\Sigma,|L_{j}|)$ and the inequality \begin{equation}
||u||_{1,p}\leq C(||u||_{L^{p}(\Sigma\setminus \cup^{k}_{l=1}B_{\frac{1}{2}}(z_{l})}+||\bar \partial u||_{p}) \end{equation}
holds if the following two conditions are satisfied:\par
$ \bullet a_{j,l}=a_{j}(h_{l})+2/p \in \mathbb{R}\setminus \mathbb{Z}$ for any $l=1,\cdots,k$\par
$ \bullet c\geq0,1<p<\infty$
\end{lemma}
$Proof$: It is totally similar to Lemma 4.8 in [10]. $\Box$ \\
Next we can similarly generalize the index decomposition theorem (Theorem 2.1) to be in case of the cylindrical metric.
\begin{theorem} For the cylindrical metric, we have \\
\\
$(1)\verb"ind"(\bar{\partial}^{t,\theta}:\hat{W}^{p}_{1,1+\kappa}\rightarrow W^{p}_{0,1+\kappa})\\
=\verb"ind"(\bar{\partial}^{t,\theta}:W^{p,B}_{1}(inn)\rightarrow W^{p}_{0}(inn))+\sum^{k}_{l=1}\verb"ind"(\bar{\partial}^{t,\theta}:\hat{W}^{p,B}_{1,1+\kappa_{j,l}}\rightarrow W^{p}_{0,1+\kappa_{j,l}})$\\
\\
$(2)\verb"ind"(\bar{\partial}^{t,\theta}:W^{p}_{1,1+\kappa}\rightarrow W^{p}_{0,1+\kappa})\\
=\verb"ind"(\bar{\partial}^{t,\theta}:W^{p,B}_{1}(inn)\rightarrow W^{p}_{0}(inn))+\sum^{k}_{l=1}\verb"ind"(\bar{\partial}^{t,\theta}:W^{p,B}_{1,1+\kappa_{j,l}}\rightarrow W^{p}_{0,1+\kappa_{j,l}})$
\end{theorem}
Now we can prove Theorem 3.1.\\
\\
$\textbf{Proof of Theorem 3.1}$: We will not repeat the same parts as Theorem 4.10 in [10], and only give the different parts. Because $c=c_{jl}=a_{j}(z_{l})\geq0$, we only need make a change related to $c=c_{jl}<0$. The most important step is as follows:\par
In case of the smooth metric, when $c=c_{j,l}<0$, we have $0<1+k_{j,l}<1$ if $2<p<2/q_{j}$, where $k_{j,l}=-c_{j,l}-2/p$,
then we deduce $W^{p}_{0,1+k}\subset W^{2}_{0}$.\par
In case of the cylindrical metric, only $c=c_{jl}=0$ can happen, so we have $0<1+k_{j,l}<1$ only simplying the condition as $p>2$.\par
 Later process is totally the same as the proof of Theorem 4.10 in [10]. $\Box$ \\
\par
Further, we can similarly get the index transformation theorem with weights and we omit the proof.
\begin{theorem}(Index transformation theorem with weights)
Under the assumption of Theorem 3.1, if $p>2$, we have the relation
\begin{eqnarray*}
&&\verb"ind"({\bar{\partial}}_{j}^{\delta}: L_{1}^{p,\delta}(\Sigma, |L_{j}|)\rightarrow L^{p,\delta}(\Sigma, |L_{j}|\otimes \wedge^{0,1}))\\
&=&\verb"ind"(\bar{\partial}_{j}^{t,\theta,\delta}: W_{1,1+\kappa+\delta}^{p}\rightarrow W_{0,1+\kappa+\delta}^{p})+\sharp\{z_{l}: c_{jl}=0\}.
\end{eqnarray*}
\end{theorem}
In the following computation in case of the cylindrical metric, we can directly apply Theorem 3.9 to prove Theorem 1.4.

\subsection{Index jumping formula}
\hspace{14pt} First by the index transformation theorem with weights (Theorem 3.9) above, we get :
\begin{eqnarray*}
&&\verb"ind"({\bar{\partial}}_{j}^{\delta}: L_{1}^{p,\delta}(\Sigma, |L_{j}|)\rightarrow L^{p,\delta}(\Sigma, |L_{j}|\otimes \wedge^{0,1}))\\
&=&\verb"ind"(\bar{\partial}_{j}^{t,\theta,\delta}: W_{1,1+\kappa+\delta}^{p}\rightarrow W_{0,1+\kappa+\delta}^{p})+\sharp\{z_{l}: c_{jl}=0\}
\end{eqnarray*}
We can focus on the $\verb"ind"(\bar{\partial}_{j}^{t,\theta,\delta}: W_{1,1+\kappa+\delta}^{p}\rightarrow W_{0,1+\kappa+\delta}^{p})$. Notice that here the Sobolev spaces $W_{1,1+\kappa+\delta}^{p}, W_{0,1+\kappa+\delta}^{p}$ are normal Sobolev spaces, so now we can apply Lockhart-McOwen theory and immediately get

\begin{theorem} Let ${\bar \partial}_{j}^{t,\theta,\delta}:W^{p}_{1,1+k+\delta}\rightarrow W^{p}_{0,1+k+\delta}$ be as above, $\forall\delta_{1}<\delta_{2}\in \mathbb{R}\setminus\mathbb{Z}$ , then we have following index jumping formula:\begin{eqnarray*}
&&\verb"ind"({\bar \partial}_{j}^{t,\theta,\delta_{1}})-\verb"ind"({\bar \partial}_{j}^{t,\theta,\delta_{2}})=N(\delta_{1},\delta_{2})=[\delta_{2}]-[\delta_{1}] \end{eqnarray*}
\end{theorem}
$Proof$: Directly apply Theorem 2.8 to the operators ${\bar \partial}^{t,\theta,\delta_{1}}, {\bar \partial}^{t,\theta,\delta_{2}}$. $\Box$ \\
\par
Let $k=1$, which means that we have only one marked point, by Theorem 3.9, Theorem 3.10, we have
\begin{theorem}
\begin{eqnarray}
\verb"ind"({\bar \partial}^{{\delta}_{1}}_{j})-\verb"ind"({\bar \partial}^{{\delta}_{2}}_{j})=[\delta_{2}]-[\delta_{1}]
\end{eqnarray}
\end{theorem}

Assume the action of orbifold line bundle $L_{j}$ at marked point $z_{l}$ is $a_{j}(h_{l})(l=1,\cdots,k)$, let $k_{j,l}=-a_{j}(h_{l})-2/p$, $k=(k_{j,1},\cdots,k_{j,k})$ as [10]. Consider corresponding weight vector $\delta=(\delta_{j,1},\cdots,\delta_{j,k}), \delta'=(\delta'_{j,1},\cdots, \delta'_{j,k})$ and weighted Sobolev space $ W^{p}_{s,1+k+\delta}, \hat{W}^{p}_{s,1+k+\delta}$\par
By Theorem 3.8, which can transform the total index jumping into the sum of computation of each end, Theorem 3.9 and Theorem 3.11, we have
\begin{theorem} In case of the cylindrical metric, we have an index jumping formula:
\begin{eqnarray}
\verb"ind"({\bar \partial}_{j}^{\delta})-\verb"ind"({\bar \partial}_{j}^{\delta'})
&=&\sum^{k}_{l=1}([\delta_{j,l}]-[\delta'_{j,l}])
\end{eqnarray}
\end{theorem}
Now we can apply results above to prove Theorem 1.4.\\
\\
$\textbf{Proof of Theorem 1.4}$: Consider the linearized operator of the Witten map \begin{eqnarray*}
D^{\delta}=D_{\wp,\mu}WI:{L_{1}}^{p,\delta}(\Sigma,L_{1}\times L_{2}\times \cdots \times L_{N})\rightarrow L^{p,\delta}(\Sigma,L_{1}\otimes\wedge^{0,1})\times\cdots L^{p,\delta}(\Sigma,L_{N}\otimes\wedge^{0,1}) \end{eqnarray*}
Note that $D^{\delta}$ is an operator over orbicurve, [12] has given the index relation between $D^{\delta}$ and
${\bar \partial}^{orb,\delta} _{j}:L^{p,\delta}_{1}(\Sigma,L_{j})\rightarrow L^{p,\delta}(\Sigma,L_{j}\otimes\wedge^{0,1})$. In fact they got :\begin{eqnarray}
\verb"ind"(D^{\delta})=\sum_{j=1}^{N}\verb"ind"({\bar \partial}^{orb,\delta} _{j})-\sum^{k}_{l=1}\sum_{j:a_{j}(h_{l})=0}1, \end{eqnarray}\\
where ${\bar \partial}^{orb,\delta} _{j}:L^{p,\delta}_{1}(\Sigma,L_{j})\rightarrow L^{p,\delta}(\Sigma,L_{j}\otimes\wedge^{0,1})$ is an operator on orbifold line bundle $L_{j}$, it is different from the operator ${\bar \partial}_{j}^{\delta}$ on desingularization $|L_{j}|$ of $L_{j}$. However, we have
\begin{theorem}

    \begin{eqnarray}\verb"ind"({\bar \partial}^{orb,\delta}_{j})=\verb"ind"({\bar \partial}_{j}^{\delta})\end{eqnarray}

\end{theorem}
$Proof$: Directly apply  Proposition 4.2.2 of [8] to our case. $\Box$ \par
Let $E=L_{j}(j=1,\cdots,N)$, where $N$ stands for the number of variable of quasi-homogeneous polynomial W in Spin equations,
then by (45) we have \begin{eqnarray*}\verb"ind"({\bar \partial}^{orb,\delta}_{j})=\verb"ind"({\bar \partial}_{j}^{\delta})\end{eqnarray*}
Let $\delta=(\delta_{jl}),\delta'=(\delta'_{jl})$ be weight matrix($j=1,\cdots,N;l=1,\cdots,k)$, and assume all weight component
 $\delta_{jl}\in \mathbb{R}\setminus\mathbb{Z}$, $\delta'_{jl}\in \mathbb{R}\setminus\mathbb{Z}$.\par
Combining (43), (44) and (45), we complete the proof of Theorem 1.4. $\Box$ \\

\section{Appendix}
\subsection{Boundary Maslov index}
\hspace{14pt} In this part, we mainly recall the definition and some properties of boundary Maslov index (see [14] for more details).\par
According to [14], first we need a special decomposition of base Riemannnian surface. Let's recall some definitions.\\
\\
$\textbf{Definition 4.1.}$ A $decomposition$ of a 2-manifold $\Sigma_{02}$ is a pair of submanifolds $\Sigma_{01}$,$\Sigma_{12}$ of $\Sigma_{02}$ such that
$\Sigma_{02}=\Sigma_{01}\bigcup\Sigma_{12}$, $\Sigma_{01}\bigcap\Sigma_{12}=\partial\Sigma_{01}\bigcap \partial\Sigma_{12}$\\
It follows that
$\partial\Sigma_{ij}=\Gamma_{i}\bigcup\Gamma_{j}$, $\Gamma_{i}\bigcap\Gamma_{j}=\emptyset$,
 where $\Gamma_{i}$ is a disjoint union of circles in $\Sigma_{02}$ and  $\Gamma_{1}=\Sigma_{01}\bigcap\Sigma_{12}$.\\
 \\
$\textbf{Definition 4.2.}$ A bundle pair $(E, F)$ over $\Sigma$ consists of a complex vector bundle $E\rightarrow\Sigma$ and a total real subbundle $F\subset E_{\partial\Sigma}$ over the boundary. A $decomposition$ of a bundle pair $(E_{02}, F_{02})$ over $\Sigma_{02}$ consists of two bundle pairs, $(E_{01}, F_{0}\cup F_{1})$ over $\Sigma_{01}$ and $(E_{12}, F_{1}\cup F_{2})$ over $\Sigma_{12}$, such that $\Sigma_{01}, \Sigma_{12}$ is a decomposition of $\Sigma_{02}$ as in definition 4.1 and $F_{i}\subset E_{02}|_{\Gamma_{i}}$.
Next we list the axiomatic definitions of boundary Maslov index:\\
\begin{theorem} There is a unique operation, called $boundary$ $Maslov$ $index$, that assigns an integer $\mu(E,F)$ to each bundle pair $(E,F)$ and satisfies the following axioms:\\
(Isomorphism): If $\Phi:E_{1}\rightarrow E_{2}$ is a vector bundle isomorphism covering a diffeomorphism $\phi:\Sigma_{1}\rightarrow\Sigma_{2}$,then  \begin{eqnarray*}
\mu(E_{1},F_{1})=\mu(E_{2},\Phi(F_{1}))  \end{eqnarray*}
(Direct sum): $\mu(E_{1}\bigoplus E_{2},F_{1}\bigoplus F_{2})=\mu(E_{1},F_{1})+\mu(E_{2},F_{2})$\\
(Composition): For a composition of a bundle pair decomposition as in definition 4.2, we have  \begin{eqnarray*}
\mu(E_{02},F_{02})=\mu(E_{01},F_{01})+\mu(E_{12},F_{12})  \end{eqnarray*}
(Normalization): For $\Sigma=D$ the unit disc, $E=D\times \mathbb{C}$ the trivial bundle and $F_{z}=e^{{ik\theta}/2}\mathbb{R}$ for $z=e^{i\theta}\in \partial D=S^1$ we have  \begin{eqnarray*}
\mu(D\times \mathbb{C},F)=k  \end{eqnarray*} \par
\end{theorem}
In addition, we have another important property which relates the boundary Maslov index to the first Chern class.\\
\begin{theorem}
If $\partial\Sigma=\emptyset$, we have \begin{eqnarray*}
\mu(E,\emptyset)=2<c_{1}(E),[\Sigma]>  \end{eqnarray*}
where $c_{1}(E)$ is the first Chern class  of $E$ and $[\Sigma]$ is the fundamental class.
\end{theorem}
\subsection{Donaldson index theory}
\hspace{14pt} We adopt notations of [9] in the following discussion.
Consider a noncompact Riemann manifold $X=Y\times \mathbb{R}$ called a tubular end, where $Y$ is a compact 3-dimensional manifold.
Let $\pi:Y\times \mathbb{R}\rightarrow Y$ be the projection, $P\rightarrow Y $ a G-principal bundle, where G is a compact Lie group, $\pi^{*}P$ fulled bundle on $X$, $A$ is a connection on $P$, we write $A$ as fulled connection on $\pi^{*}P$ if no confusion.
Assume $d_{A}$ is the corresponding convariant derivative of $A$, $d_{A}$ is its adjoint operator.\par
Donaldson defined an operator which is very important for his theory (see [9]) as follows:\begin{equation}
D_{A}=-{d^{*}}_{A}\bigoplus {d^{+}}_{A}:\Omega^{1}_{X}(\mathfrak{g}_{P})\rightarrow \Omega^{0}_{X}(\mathfrak{g}_{P})\bigoplus \Omega^{1}_{X}(\mathfrak{g}_{P}) \end{equation}
where $\mathfrak{g}_{P}$ stands for adjoint bundle of principal bundle $\pi^{*}P$ associated to the adjoint representation of $G$, $\Omega^{k}_{X}(\mathfrak{g}_{P})$ is smooth $k$-forms taking values in $g_{P}$.
What's more, Donaldson transformed $D_{A}$ to be a simple form as follows:\begin{equation}
D_{A}=\frac{d}{dt}+L: \Omega^{0}_{X}(\mathfrak{g}_{P})\bigoplus \Omega^{1}_{X}(\mathfrak{g}_{P})\rightarrow \Omega^{0}_{X}(\mathfrak{g}_{P})\bigoplus \Omega^{1}_{X}(\mathfrak{g}_{P}) \end{equation}
where $L$ is a self dual elliptic operator.
This expression is very important to his theory, that is to say, when $L$ is invertible, by separation of variables he proved the following theorem:\\
\begin{theorem}
$D_{A}=\frac{d}{dt}+L: L^{2}_{1}\rightarrow L^{2}$ is a Fredholm operator.
\end{theorem}
 Then Donaldson studied the following index problem. Assume that Riemannian manifold $X$ is a disjoint union of  two disconnected components $X=X_{1}\cup X_{2}$, and two ends $Y \times (0,\infty),{\bar Y }\times (0,\infty)$ which are identified are contained in different components.\par
Now we consider a family of Riemannian manifolds  $X^{\sharp(T)} $, depending on a real parameter $T>0$, obtained by identifying the two ends of $X$. For fixed $T$ we first delete the infinite portions $Y \times [2T,\infty),{\bar Y} \times [2T,\infty)$ from the two ends, and then identify $(y,t)\in Y\times(0,T)\subset X_{1}$ with $(y,2T-t)\in Y\times (T,2T)\subset X_{2}$. \par
This gives a connected compact Riemann manifold $X^{\sharp(T)}$. Clearly these are all diffeomorphic for different values of $T$. We will denote the manifold by $X^{\sharp}$ when the $T$ dependence is not important. The procedure is a generalization of the connected sum operation on manifolds. \par
Fix an isometry between  $Y$ and $\bar Y $. Suppose there are vector bundles $E_{1},E_{2}$ on $X_{1},X_{2}$ respectively. Consider smooth, compactly supported section spaces $\Gamma_{c}(E_{1}),\Gamma_{c}(E_{2})$ of $E_{1},E_{2}$, we define Sobolev norm as follows:\begin{eqnarray}
             ||f_{i}||_{L^{p}}:&=&(\int_{X_{i}}|f_{i}|^{p}dvol)^{1/p}, \forall f_{i}\in \Gamma_{c}(E_{i})\\
             ||\rho_{i}||_{L^{p}_{1}}:&=&(\int_{X_{i}}(|\rho_{i}|^{p}+|D_{A}\rho_{i}|^{p}) dvol)^{1/p}, \forall\rho_{i}\in \Gamma_{c}(E_{i}) \end{eqnarray}\\
             where$ |\cdot|$ denote norm induced by a Riemann metric $g$. \par Let
             $L^{p}(E_{i}),L^{p}_{1}(E_{i})$ be Sobolev spaces by completing spaces $\Gamma_{c}(E_{1}),\Gamma_{c}(E_{2})$ respectively under the form of $||\cdot||_{L^{p}},||\cdot||_{L^{p}_{1}}$ \par
Consider differential operators $D_{i}$ acting on Sobolev spaces $L^{p}_{1}(E_{i}), i=1,2$. Donaldson proved these operators were Fredholm operators in [9], so ind $D_{i}$ are well defined.\par
 There is an obvious way of constructing an bundle $E^{\sharp(T)}$ over $X^{\sharp(T)}$, identifying the bundles over the ends, and the operator $D^{\sharp(T)}$ over $X^{\sharp(T)}$. We write $D^{\sharp(T)}$ as  $D^{\sharp}$ when the $T$ dependence is not important.
\begin{theorem}[$L^{p}$-Index gluing theorem]
\begin{equation}\verb"ind"(D^{\sharp})=\verb"ind"(D_{1})+\verb"ind"(D_{2}) \end{equation}
\end{theorem}
 Two theorems above are proved under norms $L^{2}$ and $L^{2}_{1}$ , but Donaldson also pointed out these two theorems also can be extended under norms $L^{p}$ and $L^{p}_{1}$. Theorem 4.4 is almost the same as Theorem 2.5, so next we only give the proof of Theorem 2.5.\\
 \\
$\textbf{Proof of Theorem 2.5}$: We will imitate Donaldson's proof of $L^{2}$-edition step by step. The proof involves four steps, for the first three steps we suppose that the operators $D^{i}$ over $X_{i}$ have zero cokernel, then admit bounded right inverses  \begin{eqnarray*}
Q_{i}: L^{p}=L^{p}(E_{i})\rightarrow L^{p}_{1}=L^{p}_{1}(E_{i}) \end{eqnarray*}
by a following lemma, with $||Q_{i}(\rho)||_{p}=||Q_{i}(\rho)||_{L^{p}}\leq C_{p}||\rho||_{p}$ and $D_{i}Q_{i}=1$.\\

\begin{lemma}
Assume  $L$ is invertible. Consider $D=\frac{d}{dt}+L: L^{p}_{1}\rightarrow L^{p}$ has zero cokernel,
then there exists a bounded right inverse operator $Q$  and a constant $C_{p}$ such that\begin{eqnarray*}
   ||Q\rho||_{L^{p}}\leq C_{p}||\rho||_{L^{p}},DQ=1 \end{eqnarray*}
where $\rho$ is any smooth compactly supported sections of vector bundle $E_{i}$.
\end{lemma}
In $L^{2}$ case, we use separation of variables because function space is Hilbert space. When facing $L^{p}$ case, we can not use separation of variables again. Fortunately we can use the expression of inversion operator $Q$ in $L^{2}$ case to construct the inversion operator of $L^{p}$ edition as follows:\begin{eqnarray*}
     Q(\rho)=\int^{\infty}_{-\infty}K(s-t)\rho(t)dt \end{eqnarray*}
     where the kernel K is an operator-valued function.\\
$\textbf{Step 1}$  We construct, for large $T$, an injection \begin{eqnarray*}
    \alpha: ker D^{\sharp}\rightarrow ker D_{1}\oplus ker D_{2} \end{eqnarray*} \par
In fact we construct a map which is close to being an isometric embedding, with respect to the metrics on the kernels induced by the $L^{2}$ norms.
To do this we fix functions $\phi_{1},\phi_{2}$ on $\Sigma^{\sharp(T)}$ such that $\phi^{p}_{1}+\phi^{p}_{2}=1$,with $\phi_{i}$ supported in $X_{i}(3T/2)$ and such that $||\nabla \phi_{i}||_{L^{\infty}}=\epsilon(T)$,where $\epsilon(T)\rightarrow 0$ as $T\rightarrow\infty$. It is easy to write down such functions, indeed we can obviously take $\epsilon(T)=const.T^{-1}$. We then put, for $f\in ker D^{\sharp},\alpha(f)=(f_{1},f_{2})$ where \begin{eqnarray*}
 f_{i}=\phi_{i}f-Q_{i}D_{i}(\phi_{i}f).\end{eqnarray*} \par
Here we are regarding $\phi_{i}f$ as being defined over $\Sigma_{i}$ in the obvious way, using the fact that $\phi_{i}$ is a supported in $\Sigma_{i}(2T)$. The section $f_{i}$ lies in the kernel of $D_{i}$, since $Q_{i}$ is a right inverse. It is also a small perturbation of $\phi_{i}f$ in that we have \begin{eqnarray*}
       ||f_{i}-\phi_{i}f||_{p}&=&||Q_{i}D_{i}(\phi_{i}f)||_{p}\\
                              &\leq& C_{p}||D_{i}(\phi_{i}f)||_{p}\\
                              &=& C_{p}||\nabla\phi_{i}f||_{p}\\
                              &\leq& C_{p}|||\nabla\phi_{i}|f||_{p}\\
                              &\leq& C_{p}\epsilon(T)||f||_{p}.\end{eqnarray*} \par

  Here we have used the fact that $D^{\sharp}f=0$, and that $D_{i}$ can be identified with $D^{\sharp}$ over the support of $\phi_{i}$.\par
To complete the first step we now observe that, for any $f$ on $\Sigma^{\sharp(T)}$,\begin{eqnarray*}
     ||\phi_{1}f||^{p}_{p}+||\phi_{2}f||^{p}_{p}=||f||^{p}_{p}, \end{eqnarray*} \par
     since $\phi^{p}_{1}+\phi^{p}_{2}=1$. Otherwise said, the map $f\mapsto (\phi_{1}f,\phi_{2}f)$ defines an isometric embedding of $L^{p}_{\Sigma^{\sharp(T)}}$ in $L^{p}_{\Sigma^{1}}\oplus L^{p}_{\Sigma^{2}}$. This means that $\alpha$ is approximately an isometry for large $T$, precisely,\begin{eqnarray*}
 |||\alpha(f)||_{p}-||f||_{p}|&=&|||\alpha(f)||_{p}-||(\phi_{1}f,\phi_{2}f)||_{p}|\\
                                  &\leq&||(f_{1}-\phi_{1}f,f_{2}-\phi_{2}f)||_{p}\\
                                  &\leq&\sqrt[p]{2}C_{p}\epsilon(T)||f||_{p}\end{eqnarray*}\par
  where we used  estimate above and the definition of norm that $||(f_{1},f_{2})||^{p}_{p}:=||f_{1}||^{p}_{p}+||f_{2}||^{p}_{p}$ ,
     so $\alpha$ is injective once $T$ is sufficient large such that \begin{eqnarray*}\epsilon(T)<\frac{1}{\sqrt[p]{2}C_{p}} \end{eqnarray*}.\\
     \\
$\textbf{Step 2}$ We show that, under the same assumption of the existence of the right inverses $Q_{i}$, the operator $D^{\sharp}$ is also surjective for large $T$.
To do this it suffices to construct a map $P:L^{p}\rightarrow L^{p}_{1}$ over $\Sigma^{\sharp(T)}$ such that \begin{eqnarray*}
    ||D^{\sharp}P(\rho)-\rho||_{p}\leq k||\rho||_{p} \end{eqnarray*}
    where $k<1$. For then the operator $PD^{\sharp}-1$ is invertible and $Q=P(PD^{\sharp}-1)^{-1}$ is a right inverse for $D^{\sharp}$. We construct $P$ by splicing together the operators $Q_{i}$ over the individual manifolds.
    Write $\beta_{i}=\phi^{p}_{i}$,where $\phi_{i}$ are the cut-off functions above. Thus $\beta_{1}+\beta_{2}=1$. And since $0\leq \phi_{i}\leq1$, the gradient of $\beta_{i}\leq 2\epsilon(T)$.
    We define \begin{eqnarray*}
    P(\rho)=\beta_{1}Q_{1}(\rho_{1})+\beta_{2}Q_{2}(\rho_{2}) \end{eqnarray*}
    over $\Sigma^{\sharp(T)}$, where $\rho_{i}$ is the restriction of $\rho$ to $\Sigma_{i}(2T)\subset \Sigma^{\sharp(T)}$, extended by zero over the remainder of $X_{i}$. Similarly, $\beta_{i}Q_{i}(\rho_{i})$ is regarded as a section over $\Sigma^{\sharp(T)}$, extending by zero outside the support of $\beta_{i}$. Then \begin{eqnarray*}
    D^{\sharp}P(\rho)=(\beta_{1}D^{\sharp}Q_{1}(\rho_{1})+\beta_{2}D^{\sharp}Q_{2}(\rho_{2}))+((\nabla\beta_{1})\ast Q_{1}(\rho_{1})+(\nabla\beta_{2})\ast Q_{2}(\rho_{2})) \end{eqnarray*}\\
  Here, as before, we use $\ast$ to denote a certain algebraic operation. Now $\beta_{i}D^{\sharp}Q_{i}(\rho_{i})=\beta_{i}\rho_{i}=\beta_{i}\rho$, since we can identify $D^{\sharp}$ with $D_{i}$ and $\rho_{i}$ with $\rho$ over the support of $\beta_{i}$. So the first two terms in the expression above yield $\rho$ and the remainder has norm bounded by \begin{eqnarray*}
  \sum||(\nabla\beta_{i})\ast Q_{i}(\rho_{i})||_{p}\leq 4C_{p}\epsilon(T)||\rho||_{p} \end{eqnarray*}
  So \begin{eqnarray*}   ||D^{\sharp}P(\rho)-\rho||_{p}\leq 4C_{p}\epsilon(T)||\rho||_{p} \end{eqnarray*}\\and we achieve the desired "approximate inverse" by taking $T$ so large that $\epsilon(T)\leq 1/(4C_{p})$. This completes the second step in the proof.\\
  \\
$\textbf{Step 3}$ In the third step we construct, under the same assumption of the surjectivity of $D_{i}$, a linear injection $\alpha':ker D_{1}\oplus ker D_{2}\rightarrow ker D^{\sharp}$ for large enough $T$. For this we first return to the construction of the operator $Q$ above, and note that it admit an $L^{2}$ bound:\begin{eqnarray*}
    ||Q(\rho)||_{p}\leq C_{p}||\rho||_{p} \end{eqnarray*}
say, for all large enough values of $T$. With this observation the map $\alpha'$ can be constructed in a similar fashion to the map $\alpha$ in the first step.\par
For elements $f_{i}$ of the kernel of the $D_{i}$ over $\Sigma_{i}$ we set $\alpha'(f_{1},f_{2})=g-QD^{\sharp}g$, where
    $g=\beta_{1}f_{1}+\beta_{2}f_{2}$\par
Here we have identified appropriate sections over $X_{i}$ and $X$, in the way which will now be familiar to the reader. Just as in the first step, we see that the $L^{2}$ norm of the 'correction term' $QD^{\sharp}g$ is bounded by an arbitrarily small multiple of $||(f_{1},f_{2})||_{p}$.
    In fact, \begin{eqnarray*}||QD^{\sharp}g||^{p}_{p}&=&||QD^{\sharp}(\beta_{1}f_{1}+\beta_{2}f_{2})||^{p}_{p}\\
                                 &\leq& C^{p}_{p}||D^{\sharp}(\beta_{1}f_{1}+\beta_{2}f_{2})||^{p}_{p}\\
                                 &\leq& C^{p}_{p}(|||\nabla\beta_{1}|f_{1}+|\nabla\beta_{2}|f_{2}||^{p}_{p})\\
                                 &\leq& (pC_{p}\epsilon(T))^{p}(||f_{1}||^{p}_{p}+||f_{2}||^{p}_{p})\\
                                 &=& (pC_{p}\epsilon(T))^{p}(||(f_{1},f_{2})||^{p}_{p}) \end{eqnarray*} \par
    Here we use again two facts, one is  that $D_{i}f_{i}=0$ and that $D^{\sharp}$ can be identified with $D_{i}$  over the support of $\beta_{i}$, another is that $(||(f_{1},f_{2})||^{p}_{p}:=||f_{1}||^{p}_{p}+||f_{2}||^{p}_{p}$.
It remains only to show that the $L^{p}$ norm of $g$ is close to that of $(f_{1},f_{2})$ for large $T$.\par
    By integrability of $f_{i}$, for any $\eta=\eta(T)>0$, we can choose sufficient large $T_{0}$ such that for $T\geq T_{0}$ then,\begin{eqnarray*}
       \int_{(T/2,\infty)}|f_{i}|^{p}\leq\eta||f_{i}||^{p}_{p}
       , \forall f_{i}\in ker D_{i}.\end{eqnarray*} \par
       Since $\beta_{i}=1$ on the segment $Y\times (0,T/2)$ of the tube in $\Sigma_{i}$, we clearly have \begin{eqnarray*}
       ||g||^{p}_{p}&=&||\beta_{1}f_{1}+\beta_{2}f_{2}||^{p}_{p}\geq\int_{Y\times(0,T/2)}|f_{1}|^{p}+\int_{{\bar Y}\times(0,T/2)}|f_{2}|^{p}\\
                                                &=&||f_{1}||^{p}_{p}+||f_{2}||^{p}_{p}-\int_{Y\times(T/2,\infty)}|f_{1}|^{p}+\int_{{\bar Y}\times(T/2,\infty)}|f_{2}|^{p}\\
                                                &\geq&(1-\eta)(||f_{1}||^{p}_{p}+||f_{2}||^{p}_{p})\\
                                                &=&(1-\eta)||(f_{1},f_{2})||^{p}_{p}). \end{eqnarray*}\par
       And obviously $||g||^{p}_{p}=||\beta_{1}f_{1}+\beta_{2}f_{2}||^{p}_{p}\leq||(f_{1},f_{2})||^{p}_{p}$,
     so we have \begin{eqnarray*} |||(f_{1},f_{2})||^{p}_{p}-||\alpha'(f_{1},f_{2})||^{p}_{p}|&=&|||(f_{1},f_{2})||^{p}_{p}-||g-QD^{\sharp}g||^{p}_{p}|\\
     &\leq&||(f_{1},f_{2})||^{p}_{p}-(||g||^{p}_{p}-||QD^{\sharp}g||^{p}_{p})\\
     &\leq& (\eta(T)+(pC_{p}\epsilon(T))^{p})(||(f_{1},f_{2})||^{p}_{p}) \end{eqnarray*} \par
          Therefore, when $\alpha'(f_{1},f_{2})=0$, we can choose sufficient large $T$ such that $(\eta(T)+(pC_{p}\epsilon(T))^{p})<1$, then by inequality above, we get $||(f_{1},f_{2})||^{p}_{p}=0\Rightarrow (f_{1},f_{2})=(0,0)$, so we prove that $\alpha'$ is an injection.\par
 These first three steps complete the proof of the "gluing formula" in the case when the $D_{i}$ are surjective.
    For in this case we have, by step 2, ind($D_{i}$)=dim ker$D_{i}$, ind($D^{\sharp}$)=dim ker$D^{\sharp}$ for large $T$.
    By step 1, dim ker $D^{\sharp}\leq$dim ker $D_{1}$+dim ker $D_{2}$, and step 3 gives the reverse inequality, so dim ker $D^{\sharp}=$dim ker $D_{1}$+dim ker $D_{2}$ as required. Hence \begin{eqnarray*}
                              \verb"ind"(D^{\sharp})=\verb"ind"(D_{1})+\verb"ind"(D_{2})\end{eqnarray*}\\
$\textbf{Step 4 }$ This step is totally the same as Donaldson's proof of $L^{2}$-edition. First we remove the assumption that the operators $D_{i}$ are surjective. We do this by modifying the operators.
Assume that $n_{i}=$dim coker $D_{i}$, this is well defined because $D_{i}$ are Fredholm operators.
We can choose injective maps $U_{i}:R^{n_{i}}\rightarrow \Gamma(E_{i})$ with images supported in the interior of the $\Sigma_{i}$, and such that \begin{eqnarray*}
      \tilde{D_{i}}\equiv D_{i}\oplus U_{i}: \Gamma(E_{i})\oplus R^{n_{i}}\rightarrow  \Gamma(E_{i}) \end{eqnarray*}
      is surjective. We have\begin{eqnarray*}
       \verb"ind"\tilde{D_{i}} =\verb"ind "D_{i}+\verb"dim" \textit{coker}D_{i}=\verb"ind "D_{i}+n_{i} \end{eqnarray*}

      We can form an obvious operator $\tilde{D^{\sharp}}=D^{\sharp}\oplus U_{1}\oplus U_{2}$ over $\Sigma^{\sharp(T)}$, and the proof above goes without any change to show that
                $\verb"ind"\tilde{D^{^{\sharp}}}=\verb"ind"\tilde{D_{1}}+\verb"ind"\tilde{D_{2}}$, so \begin{eqnarray*}
                \verb"ind"D^{\sharp}=\verb"ind"\tilde{D^{\sharp}}-(n_{1}+n_{2})
                              =\verb"ind"\tilde{D_{1}}+\verb"ind"\tilde{D_{2}}-(n_{1}+n_{2})\\
                              =(\verb"ind"\tilde{D_{1}}-n_{1})+(\verb"ind"\tilde{D_{2}}-n_{2})
                              =\verb"ind"D_{1}+\verb"ind"D_{2} \end{eqnarray*}
                              We ultimately complete the proof of Theorem 2.5. $\Box$.\\
                \par
When $L$ is not invertible, Donaldson considered weight Sobolev space $L^{2,\alpha}$ and $L^{2,\alpha}_{1}$, which are defined as follows:
\begin{equation}
||f||_{L^{2,\alpha}_{1}}=||e^{\alpha t}f||_{L^{2}_{1}}, ||f||_{L^{2,\alpha}}=||e^{\alpha t}f||_{L^{2}} \end{equation}
where $\alpha\in\mathbb{R}$.
Then consider $D_{A}=\frac{d}{dt}+L: L^{2,\alpha}_{1}\rightarrow L^{2,\alpha}$, because multiple operation of $e^{\alpha t}$ is an isometry from $L^{2,\alpha}_{1}$ and $L^{2,\alpha}$ to $L^{2}_{1}$ and $L^{2}$ separately.
Therefore, $D_{A}=\frac{d}{dt}+L: L^{2,\alpha}_{1}\rightarrow L^{2,\alpha}$ is equivalent to:\begin{equation}
 e^{\alpha t}D_{A}e^{-\alpha t}=\frac{d}{dt}+(L-\alpha): L^{2}_{1}\rightarrow L^{2} \end{equation} \par
We can see the introduction of weighted Sobolev space is equivalent to replace the operator $L$ of $(L-\alpha)$, then theorems without weights above can be easily to generalized  in weighted Sobolev space.\par
Now let's consider the gluing problem with weights. Suppose we are in the situation above, we consider $D_{A}=D+\nu:\frac{d}{dt}+L: L^{p,\alpha}_{1}\rightarrow L^{p,\alpha}$\\
where $\nu$ is an algebraic operator, represented as multiplication by $\alpha_{i},i=1,2$ in our description over the $i$th end. The operator $D$ is represented as $\frac{d}{dt_{1}}+L_{Y}$ over the first end and as $\frac{d}{dt_{2}}+L_{\bar Y}$ over the second. When we identify the ends to form $\Sigma^{\sharp}$ we reverse the time co-ordinates, so $\frac{d}{dt_{1}}$ corresponds to $-\frac{d}{dt_{2}}$, and this marries up with the natual identification $L_{Y}=-L_{\bar Y}$. Thus, in our gluing operation, the operator \begin{eqnarray*}
\frac{d}{dt_{1}}+L_{Y}+\alpha_{1} \end{eqnarray*}
over the first end can naturally be identified with \begin{eqnarray*}
\frac{d}{dt_{2}}+L_{\bar Y}+\alpha_{2} \end{eqnarray*}
over the second if $\alpha_{2}=-\alpha_{1}$, now we still write $D^{\sharp}$ as the operator after gluing. \par
In this case the arguments we used before go through without any change to show that
\begin{theorem}
Assuming $\alpha\in \mathbb{R} $ such that $L-\alpha$ is invertible, we have \begin{eqnarray*}
\verb"ind"(D^{\sharp})=\verb"ind"(D^{\alpha_{1}})+\verb"ind"(D^{-\alpha_{1}}) \end{eqnarray*}
\end{theorem}
We can further consider weight vector case which corresponds to more ends. Choose a weights $\alpha_{i}$ for each end $Y_{i}\times (0,\infty)$ of $\Sigma$. Fix a positive function $W$ on $\Sigma$ which is equal to $e^{\alpha_{i}t}$ on the $i$th end and define norms: \begin{eqnarray*}
   ||f||_{L^{p,\vec{\alpha}}}=||Wf||_{L^{p}}, ||f||_{L_{1}^{p,\vec{\alpha}}}=||Wf||_{L_{1}^{p}} \end{eqnarray*}
   with completions $L^{p,\vec{\alpha}},L_{1}^{p,\vec{\alpha}}$. Different choices of $W$, with the same
   weight vector $\vec{\alpha}=(\alpha_{1},\alpha_{2},\cdots,\alpha_{N} )$, give equivalent norms.\par
Similarly, we can easily obtain index gluing formula in weight vector case. \\
\begin{theorem}
Assuming $\vec{\alpha}\in \mathbb{R}^{N} $ such that $L-\vec{\alpha}$ is invertible, we have
\begin{eqnarray*}
\verb"ind"(D^{\sharp;(\alpha_{2},\cdots,\alpha_{N})})=\verb"ind"(D^{(\alpha_{1},\alpha_{2},\cdots,\alpha_{N})})
+\verb"ind"(D^{(-\alpha_{1},\alpha_{2},\cdots,\alpha_{N})})\end{eqnarray*}
\end{theorem}
\section{Acknowledgements}

The author would like to thank his advisor Professor Huijun Fan for introducing him to the subject, helpful discussions and constant encouragement.\\

\end{document}